\documentclass[12pt,reqno]{amsart}
\usepackage{amsmath,amssymb,extarrows}
\usepackage{url}
\usepackage{tikz,enumerate}
\usepackage{diagbox}
\usepackage{appendix}
\usepackage{epic}

\usepackage{float} 

\usepackage{cite}

\usepackage{hyperref}
\usepackage{array}

\usepackage{booktabs}

\allowdisplaybreaks[4]

\newtheorem{theorem}{Theorem}
\newtheorem{corollary}[theorem]{Corollary}

\newtheorem{lemma}[theorem]{Lemma}

\theoremstyle{definition}

\newtheorem{exam}{Example}
\theoremstyle{remark}
\newtheorem{rem}{Remark}
\numberwithin{equation}{section}
\numberwithin{theorem}{section}
\numberwithin{defn}{section}

\newcommand{\abs}[1]{\left\vert#1\right\vert}

\newcommand{\diag}{\mathrm{diag}}

\begin{document}
\title[Modularity of Some Nahm sums as Vector-valued Functions]
 {Modularity of Some Nahm sums as Vector-valued Functions}

\author{Liuquan Wang and Huohong Zhang}
\address[L. Wang]{School of Mathematics and Statistics, Wuhan University, Wuhan 430072, Hubei, People's Republic of China}
\email{wanglq@whu.edu.cn;mathlqwang@163.com}

\address[H. Zhang]{School of Mathematics and Statistics, Wuhan University, Wuhan 430072, Hubei, People's Republic of China}
\email{hhzhang00@whu.edu.cn}

\dedicatory{Dedicated to George Andrews and Bruce Berndt for their 85th birthdays}



\begin{abstract}
Zagier observed that modular Nahm sums associated with the same matrix may form a vector-valued modular function on some congruence subgroup. We establish modular transformation formulas for several families of Nahm sums by viewing them as vector-valued functions, and thereby we show that they are indeed modular on the congruence subgroup $\Gamma_0(N)$ with $N=1,2,3,4$.  In particular, we prove two transformation formulas discovered by Mizuno  related to the Kanade--Russell mod 9 conjecture and Capparelli's identities. We also establish vector-valued transformation formulas for some theta series. As applications, we give modular transformation formulas for various families of Nahm sums involving those in the Andrews--Gordon identities and Bressoud's identities.
\end{abstract}

\maketitle

\section{Introduction and Main Results}\label{sec-intro}

The famous Rogers--Ramanujan identities, first proved by Rogers \cite{R-R-Rogers} and later rediscovered by Ramanujan, state that
\begin{align}
G(q):=\sum_{n=0}^{\infty} \frac{q^{n^2}}{(q;q)_n} = \frac{1}{(q,q^4;q^5)_\infty}, \label{R-R-1} \\
H(q):=\sum_{n=0}^{\infty} \frac{q^{n^2+n}}{(q;q)_n} = \frac{1}{(q^2,q^3;q^5)_\infty}. \label{R-R-2}
\end{align}
Here and throughout this paper,  we assume that $\abs{q} < 1$ and use standard $q$-series notations:
\begin{align}
(a;q)_0:=1, \quad (a;q)_n:=\prod_{k=0}^{n-1}(1-aq^k), \quad (a;q)_\infty :=\prod_{k=0}^{\infty}(1-aq^k),  \\
 (a_1,\dots,a_m;q)_n:=(a_1;q)_n\cdots (a_m;q)_n ,~ n\in \mathbb{N}\cup \{ \infty\}.
\end{align}

The Andrews--Gordon identity \cite{Andrews1974} generalizes the above identities to arbitrary odd moduli. It asserts that for any integer $k\geq 2$ and $1\leq i\leq k$,
\begin{align}\label{eq-AG}
x_{2k+1,i}(q)&:=\sum_{n_{1},\dots,n_{k-1}\geq 0}\frac{q^{N_{1}^2+\cdots+N_{k-1}^2 +N_{i}+\cdots+N_{k-1}}}{(q;q)_{n_{1}} \cdots (q;q)_{n_{k-2}} (q;q)_{n_{k-1}}} \nonumber \\
&= \frac{(q^{i},q^{2k+1-i},q^{2k+1};q^{2k+1})_\infty}{(q;q)_\infty},
\end{align}
where $N_j = n_j+\cdots+n_{k-1}$ if $j\leq k-1$ and $N_k = 0$. As for the even moduli case, Bressoud's identity \cite[Eq.\ (3.4)]{Bressoud1980} states that for any integer $k\geq 2$ and $1\leq i\leq k$,
\begin{align}
x_{2k,i}(q)&:=\sum_{n_{1},\dots,n_{k-1}\geq 0}\frac{q^{N_{1}^2+\cdots+N_{k-1}^2 +N_{i}+\cdots+N_{k-1}}}{(q;q)_{n_{1}} \cdots (q;q)_{n_{k-2}} (q^2;q^2)_{n_{k-1}}} \nonumber \\
&= \frac{(q^{i},q^{2k-i},q^{2k};q^{2k})_\infty}{(q;q)_\infty},  \label{B-identity}
\end{align}
where $N_j$ is defined as before.

Now there are many sum-to-product identities similar to \eqref{R-R-1}--\eqref{R-R-2} and \eqref{eq-AG}--\eqref{B-identity} and they were called as Rogers--Ramanujan type identities. Surprisingly, there are deep connections between these identities and various branches of mathematics such as number theory, combinatorics and Lie algebras and they are also closely related to mathematical physics such as rational conformal field theories.

An important problem linking the theory of $q$-series and modular forms is to understand the modularity of $q$-hypergeometric series. Rogers--Ramanujan type identities serve as an useful tool to this problem. These identities turn $q$-hypergeometric series to some nice infinite products which reveal the modularity of the original series. For example, the infinite products on the right sides of \eqref{R-R-1} and \eqref{R-R-2} are reciprocals of the generalized Dedekind eta functions and hence are modular. This fact is not easy to be captured  from the left sides.

In the direction of studying modularity of $q$-series, Nahm considered the following particular important class of series motivated from physics:
\begin{equation}
f_{A,B,C}(q) := \sum_{n=(n_1,\dots,n_r)^{\mathrm{T}}\in (\mathbb{Z}_{\geq 0})^{r}}\frac{q^{\frac{1}{2} n^{\rm T}An+n^{\rm T}B+C}}{(q;q)_{n_{1}}\cdots (q;q)_{n_r}},  \label{Nahm-def}
\end{equation}
where $r\geq 1$ is a positive integer, $A$ is a real positive definite symmetric $r\times r$ matrix, $B$ is a column vector of length $r$, and $C$ is a scalar. The series $f_{A,B,C}(q)$ is called a \emph{Nahm sum} or \emph{Nahm series}. Nahm \cite{Nahm-question} proposed the following famous problem: find all such $(A, B, C)$ with rational entries for which the series $f_{A,B,C}(q)$ becomes modular, and such $(A,B,C)$ is said to be a rank $r$ \emph{modular triple}. These modular Nahm sums are expected to be characters of some rational conformal field theories. For example, the Rogers--Ramanujan identities \eqref{R-R-1} and \eqref{R-R-2} show that  $(2,0,-1/60)$ and $(2,1,11/60)$ are two rank one modular triples.

In 2007, Zagier \cite{Zagier2007} made important progress towards Nahm's problem. He proved that there are exactly seven modular triples when the rank $r=1$. When the rank $r=2$ and $3$, Zagier \cite{Zagier2007} found many possible modular triples. After the works  of Zagier \cite{Zagier2007}, Vlasenko--Zwegers \cite{VZ}, Cherednik--Feigin \cite{Feigin}, Wang \cite{Wang2022,W-rank3} and Cao--Rosengren--Wang \cite{C-R-W}, the modularity of all of Zagier's rank two and rank three examples have been confirmed.
Zagier \cite{Zagier2007} also stated explicitly a conjecture attributed to Nahm \cite{Nahm-conj}, which provides sufficient and necessary conditions about the matrix $A$ so that it is the matrix part of some modular triples.

In 2023, Mizuno \cite{Mizuno} considered the modularity of the generalized  Nahm sums defined by
\begin{align}\label{eq-general-Nahm-sum}
   \widetilde{f}_{A,B,C,D}(q):= \sum_{n=(n_1,\dots,n_r)^\mathrm{T} \in (\mathbb{Z}_{\geq 0})^r} \frac{q^{\frac{1}{2}n^\mathrm{T}ADn+n^\mathrm{T}B+C}}{(q^{d_1};q^{d_1})_{n_1}\cdots (q^{d_r};q^{d_r})_{n_r}}.
\end{align}
Here $D=\mathrm{diag}(d_1,\dots,d_r)$ ($d_1,\dots,d_r \in \mathbb{Z}_{>0}$),  $A \in \mathbb{Q}^{r \times r} $ is a symmetrizable matrix with the symmetrizer $D$, i.e., $AD$ is symmetric positive definite, $B \in \mathbb{Q}^r$ is a vector and  $C \in \mathbb{Q}$ is a scalar. Such generalized Nahm sums appeared frequently in sum-sides of Rogers--Ramanujan type identities. One instance is Bressoud's identities \eqref{B-identity}. Another famous example is Capparelli's partition identities \cite{Capparelli}:
\begin{align}
a_1(q)&:=\sum_{n_1,n_2\geq 0} \frac{q^{2n_1^2+6n_1n_2+6n_2^2}}{(q;q)_{n_1}(q^3;q^3)_{n_2}}=(-q^2,-q^3,-q^4,-q^6;q^6)_\infty, \label{eq-Capparelli-1} \\
a_2(q)&:=\sum_{n_1,n_2\geq 0} \frac{q^{2n_1^2+6n_1n_2+6n_2^2+n_1+3n_2}(1+q^{2n_1+3n_2+1})}{(q;q)_{n_1}(q^3;q^3)_{n_2}} \nonumber \\
&=(-q,-q^3,-q^5,-q^6;q^6)_\infty. \label{eq-Capparelli-2}
\end{align}
They were discovered by the theory of affine Lie algebras and initially stated as partition identities. The above Nahm series representations of these identities were found independently by Kanade--Russell \cite{KR2019} and Kur\c{s}ung\"{o}z \cite{Kursungoz}. As before, when $\widetilde{f}_{A,B,C,D}(q)$ is modular we call $(A,B,C,D)$ a rank $r$ \emph{modular quadruple}.

Mizuno \cite{Mizuno} provided $14$ sets of possible rank two modular quadruples, and he also proposed $19$ and $15$ sets of possible rank three modular quadruples with symmetrizers $\diag (1, 1, 2)$ and $\diag (1, 2, 2)$ respectively. Besides two rank two examples, the modularity of all other examples have been confirmed in a series of works by Wang--Wang \cite{WWg2,WW112,WW122}.  One of the remaining unsolved examples corresponds to the Kanade--Russell mod $9$ conjecture \cite{K-R-conj} (in the form by Kur\c{s}ung\"oz\cite{Kursungoz-2}):
\begin{align}
b_1(q) := \sum_{n_1,n_2\geq 0} \frac{q^{n_1^2+3n_1n_2+3n_2^2}}{(q;q)_{n_1}(q^3;q^3)_{n_2}} = \frac{1}{(q,q^3,q^6,q^8;q^9)_\infty},  \label{KR-a1} \\
b_2(q) := \sum_{n_1,n_2\geq 0} \frac{q^{n_1^2+3n_1n_2+3n_2^2+n_1+3n_2}}{(q;q)_{n_1}(q^3;q^3)_{n_2}} = \frac{1}{(q^2,q^3,q^6,q^7;q^9)_\infty},  \label{KR-a2} \\
b_3(q) := \sum_{n_1,n_2\geq 0} \frac{q^{n_1^2+3n_1n_2+3n_2^2+2n_1+3n_2}}{(q;q)_{n_1}(q^3;q^3)_{n_2}} = \frac{1}{(q^3,q^4,q^5,q^6;q^9)_\infty}. \label{KR-a3}
\end{align}

As predicated by physics, Zagier \cite{Zagier2007} made the following important observation. If the matrix $A$ satisfies the condition stated in Nahm's conjecture,  then the collection of modular Nahm sums associated with $A$ span a vector space which is invariant under the action of $\mathrm{SL}(2,\mathbb{Z})$ (bosonic case) or at least $\Gamma(2)$ (fermionic case). Here we recall the full modular group
\begin{align}
\mathrm{SL}(2,\mathbb{Z}):=\left\{\begin{pmatrix} a & b \\ c & d \end{pmatrix}: ad-bc=1,a,b,c,d\in \mathbb{Z}\right\}
\end{align}
and some of its congruence subgroups:
\begin{align}
\Gamma(N):=\left\{\gamma \in \mathrm{SL}(2,\mathbb{Z}):\gamma \equiv \begin{pmatrix} 1 & 0 \\ 0 & 1 \end{pmatrix} \pmod{N} \right\}, \\
\Gamma_0(N):=\left\{\gamma \in \mathrm{SL}(2,\mathbb{Z}):\gamma \equiv \begin{pmatrix} * & * \\ 0 & * \end{pmatrix} \pmod{N} \right\}.
\end{align}
Zagier's observation is also motivated from asymptotic analysis as discussed in \cite[Section B]{Zagier2007}.

A typical example supporting the above observation is given by the Rogers--Ramanujan identities. Let $q=e^{2\pi i \tau}$ ($\mathrm{Im}~ \tau>0$) and $\zeta_N=e^{2\pi i/N}$ throughout this paper. If we combine the functions in \eqref{R-R-1} and \eqref{R-R-2} into a single vector-valued function
\begin{align}\label{RR-vector}
    X(\tau):=(q^{-1/60}G(q),q^{11/60}H(q))^\mathrm{T},
\end{align}
then we have transformation formulas with respect to the full modular group:
\begin{align}
    X(\tau+1)=\begin{pmatrix}
        \zeta_{60}^{-1} & 0 \\ 0 & \zeta_{60}^{11}
    \end{pmatrix} X(\tau), \quad X\left(-\frac{1}{\tau}\right)=\frac{2}{\sqrt{5}}\begin{pmatrix}
        \sin \frac{2\pi}{5} & \sin \frac{\pi}{5} \\
        \sin\frac{\pi}{5} & -\sin \frac{2\pi}{5}     \end{pmatrix}X(\tau).
\end{align}
In a recent work, Milas--Wang \cite{Milas-Wang} established similar transformation formulas for a vector-valued function consisting of generalized tadpole Nahm sums.

Mizuno \cite[Eq.\ (9)]{Mizuno} discovered the following transformation formulas without proof.
\begin{theorem}\label{thm-M-conj}
Let $b_i(q)$ ($i=1,2,3$) be defined as in \eqref{KR-a1}--\eqref{KR-a3} and
$$X(\tau) := (q^{-1/18}b_1(q),q^{5/18}b_2(q),q^{11/18}b_3(q))^\mathrm{T}.$$
We have
\begin{align}
X(\tau+1) &= \mathrm{diag}(\zeta_{18}^{-1},\zeta_{18}^{5},\zeta_{18}^{11})
X(\tau), \label{X-KR-T} \\
X\Big(-\frac{1}{\tau}\Big) &= \begin{pmatrix}
\alpha_1 & \alpha_2 & \alpha_4 \\
\alpha_2 & -\alpha_4 & -\alpha_1 \\
\alpha_4 & -\alpha_1 & \alpha_2
\end{pmatrix}
X\Big(\frac{\tau}{3}\Big) \label{goal}
\end{align}
where $\alpha_k = \frac{1}{2\sqrt{3}\sin \frac{k\pi}{9}}$. As a consequence, $X(\tau)$ is a vector-valued modular function on $\Gamma_0(3)$.
\end{theorem}
Mizuno \cite[Eqs.\ (45), (54)]{Mizuno} also conjectured two vector-valued transformation formulas between rank three Nahm sums, and they have now been confirmed by Wang--Wang \cite{WW112,WW122}. However, the method in \cite{WW112,WW122} does not work for proving \eqref{goal}. The first goal of this paper is to give a proof for it. Interestingly, our proof will employ Garvan's 3-dissection formulas \cite{Garvan} for the generating function of partition cranks.

Mizuno \footnote{Private communication via email dated on August 30, 2023.}  informed the first author that he conjectured a formula for the Nahm sums involved in Capparelli's identities \eqref{eq-Capparelli-1}--\eqref{eq-Capparelli-2}. The second goal of this paper is to present a proof for his conjecture which we state as
\begin{theorem}\label{thm-M-Capparelli-conj}
Let $X(\tau)=(q^{-1/24}a_1(q),q^{5/24}a_2(q))^\mathrm{T}$ where $a_1(q)$ and $a_2(q)$ are defined in \eqref{eq-Capparelli-1} and \eqref{eq-Capparelli-2}. We have
\begin{align}
X(\tau+1)&=\diag(\zeta_{24}^{-1},\zeta_{24}^5)X(\tau), \label{X-Capparelli-T} \\
X\left(-\frac{1}{\tau}\right)&=\frac{1}{\sqrt{2}}\begin{pmatrix} 1 & 1 \\ 1 & -1 \end{pmatrix} X\left(\frac{\tau}{3}\right). \label{X-Capparelli-S}
\end{align}
As a consequence, $X(\tau)$ is a vector-valued modular function on $\Gamma_0(3)$.
\end{theorem}

Motivated by the above works \cite{Milas-Wang,Mizuno,WWg2,WW112,WW122,Zagier2007}, the third goal of this paper is to establish similar transformation formulas for other Nahm sums. We combine some Nahm sums associated with the same matrix into a single vector-valued function, and try to find their modular transformation formulas. On the one hand, this helps us to understand the interrelation between different Nahm sums. On the other hand, this provides more explicit examples of vector-valued modular functions which are relatively rare in the literature.

As one of the main results, we find transformation formulas for the  Nahm sums appeared in the Andrews--Gordon identities \eqref{eq-AG} and Bressoud's identities \eqref{B-identity}. Note that the numerators in these identities are closely related to  the theta series \cite[p.\ 215]{Wakimoto}
\begin{align}
    g_{j,m}(\tau):=\sum_{k\in \mathbb{Z}} (-1)^kq^{m(k+\frac{j}{2m})^2}=q^{j^2/4m}(q^{m+j},q^{m-j},q^{2m};q^{2m})_\infty, \label{g-defn} \\
    h_{j,m}(\tau):=\sum_{k\in \mathbb{Z}} q^{m(k+\frac{j}{2m})^2}=q^{j^2/4m}(-q^{m-j},-q^{m+j},q^{2m};q^{2m})_\infty. \label{hg}
\end{align}
where  $m\in \frac{1}{2}\mathbb{N}$ and $j\in \frac{1}{2}\mathbb{Z}$ and the last equality follows from the Jacobi triple product identity (see \eqref{JTP}). We first study transformation properties satisfied by the vector
\begin{align}
G_m(\tau):=(g_{0,m}(\tau), g_{1,m}(\tau),g_{2,m}(\tau),\dots, g_{m-1,m}(\tau))^\mathrm{T}.
\end{align}
We split this vector into two parts corresponding to odd and even values of $j$, and we obtain the following results for each part.
\begin{theorem}   \label{thm-G-o}
For integer $k\geq 2$, denote the largest odd number less than $k$ by $t_{k}$. Let
\begin{equation}
G_{k}^{(1)}(\tau) := \big(g_{1,k}(\tau), g_{3,k}(\tau), \dots, g_{t_k,k}(\tau)\big)^{\mathrm{T}},
\end{equation}
then we have
\begin{align}\label{Gk-odd-tran}
G_{k}^{(1)}(\tau+1) &= \Lambda_kG_{k}^{(1)}(\tau),  \quad
G_{k}^{(1)}\Big(-\frac{1}{4\tau}\Big) = \frac{2(-2i\tau)^{\frac{1}{2}}}{\sqrt{k}} A_{k} G_{k}^{(1)}(\tau),
\end{align}
where $A_k=(a_{ij})$ is a $\lfloor k/2 \rfloor \times \lfloor k/2 \rfloor$ matrix with entry $a_{ij} =  \cos \frac{(2i-1)(2j-1)\pi}{2k}$ and
\begin{align}
\Lambda_{k} := & \diag (e^{\frac{1}{2k}\pi i}, e^{\frac{9}{2k}\pi i}, \dots, e^{\frac{t_{k}^2}{2k}\pi i}).
\end{align}
As a consequence, we have
\begin{align}\label{Gk-odd-4}
G_{k}^{(1)}\Big(\frac{\tau}{4\tau+1}\Big) &= \frac{4\sqrt{4\tau+1}}{k} A_{k}\Lambda_{k}^{-1}A_{k}G_{k}^{(1)}(\tau).
\end{align}
\end{theorem}

\begin{theorem}  \label{thm-G-e}
For integer $k\geq 2$, denote the largest even number less than $k$ by $s_{k}$ and the largest odd number not exceeding $k$ by $\widetilde{t}_{k}$. Let
\begin{align}
G_{k}^{(0)}(\tau) := \big(g_{0,k}(\tau), g_{2,k}(\tau), \dots, g_{s_k,k}(\tau)\big)^{\mathrm{T}}.
\end{align}
We have
\begin{align}
G_{k}^{(0)}(\tau+1) & = \diag (1, e^{\frac{2}{k}\pi i}, \dots, e^{\frac{s_{k}^{2}}{2k}\pi i})G_{k}^{(0)}(\tau), \label{G_k-1} \\
G_{k}^{(0)}\Big(\frac{\tau}{4\tau+1}\Big) &= \frac{4\sqrt{4\tau+1}}{k} B_{k} \widetilde{\Lambda}_{k}C_{k}G_{k}^{(0)}(\tau), \label{G_k-2}
\end{align}
where $\widetilde{\Lambda}_{k} :=  \diag (e^{-\frac{1}{2k}\pi i}, e^{-\frac{9}{2k}\pi i}, \dots, e^{-\frac{\widetilde{t}_{k}^2}{2k}\pi i})$, $B_k=(b_{ij})$ and $C_k=(c_{ij})$ are $\lfloor{(k+1)/2}\rfloor$ $\times \lfloor{(k+1)/2} \rfloor$ matrices with
\begin{align}
&b_{ij} := \begin{cases}
\cos \frac{(i-1)(2j-1)\pi}{k}, &   1\leq i,j\leq \frac{k}{2} ~~\text{and $k$ is even}, \\
\cos \frac{(i-1)(2j-1)\pi}{k}, & 1\leq j\leq \frac{k-1}{2} ~~\text{and $k$ is odd},  \\
\frac{1}{2}\cos (i-1)\pi, & j=\frac{k+1}{2}, ~~\text{and $k$ is odd},
\end{cases}  \\
    &c_{ij} :=  \begin{cases}
\frac{1}{2}, & j=1, \\
\cos \frac{(2i-1)(j-1)\pi}{k}, & 2\leq j\leq \lfloor\frac{k+1}{2}\rfloor.
\end{cases}
\end{align}
\end{theorem}

Theorems \ref{thm-G-o} and \ref{thm-G-e} show that $G_k^{(i)}(\tau)$ ($i=0,1$) hence $G_k(\tau)$ are vector-valued modular forms on $\Gamma_0(4)$ of weight $1/2$. These theorems are applicable to a wide range of Nahm sums which can be expressed by infinite products involving $g_{j,k}(\tau)$. In particular, we obtain modular transformation formulas for the Nahm sums $x_{k,i}(q)$ defined in \eqref{eq-AG} and \eqref{B-identity}.
\begin{corollary}  \label{cor-AG}
For any integer $k\geq 2$, let
\begin{align}
&X_{2k+1}(\tau) :=q^{-\frac{1}{24}}\Big(q^{\frac{1}{8(2k+1)}}x_{2k+1,k}(q),q^{\frac{3^2}{8(2k+1)}} x_{2k+1,k-1}(q), \dots, q^{\frac{(2k-1)^2}{8(2k+1)}}x_{2k+1,1}(q)\Big)^{\mathrm{T}}.
\end{align}
We have
\begin{align}
X_{2k+1}(\tau+1) = & e^{-\frac{1}{12}\pi i}\diag (e^{\frac{1}{4(2k+1)}\pi i}, e^{\frac{9}{4(2k+1)}\pi i}, \dots, e^{\frac{(2k-1)^2}{4(2k+1)}\pi i})X_{2k+1}(\tau), \label{AG-T} \\
X_{2k+1}\left(-\frac{1}{\tau}\right) = & \frac{2}{\sqrt{2k+1}}
A_{2k+1}
X_{2k+1}(\tau), \label{AG-S}
\end{align}
where $A_{2k+1}$ is defined in Theorem \ref{thm-G-o}.
As a consequence, $X_{2k+1}(\tau)$ is a vector-valued modular function on  $\mathrm{SL}(2,\mathbb{Z})$.
\end{corollary}

\begin{corollary}  \label{cor-B}
For any integer $k\geq 2$, we define
\begin{align}
&X_{2k}^{(0)}(\tau):=q^{-\frac{1}{24}}\Big(x_{2k,k}(q),q^{\frac{2^2}{4k}}x_{2k,k-2}(q),\dots,q^{\frac{s_k^2}{4k}}x_{2k,k-s_k}(q)\Big)^\mathrm{T}, \\
&X_{2k}^{(1)}(\tau):=q^{-\frac{1}{24}}\Big(q^{\frac{1}{4k}}x_{2k,k-1}(q),q^{\frac{3^2}{4k}}x_{2k,k-3}(q),\dots,q^{\frac{t_k^2}{4k}}x_{2k,k-t_k}(q)\Big)^\mathrm{T}.
\end{align}
We have
\begin{align}
X_{2k}^{(0)}(\tau+1) &=e^{-\frac{1}{12}\pi i} \diag \big(1, e^{\frac{2^2}{2k}\pi i}, \dots, e^{\frac{s_{k}^2}{2k}\pi i}\big) X_{2k}^{(0)}(\tau), \\
X_{2k}^{(1)}(\tau+1) &= e^{-\frac{1}{12}\pi i} \diag \big(e^{\frac{1}{2k}\pi i}, e^{\frac{3^2}{2k}\pi i}, \dots, e^{\frac{t_{k}^2}{2k}\pi i}\big) X_{2k}^{(1)}(\tau), \\
X_{2k}^{(0)}\bigg(\frac{\tau}{4\tau+1}\bigg) &= \frac{2}{k}e^{\frac{\pi i}{3}} B_{k}\widetilde{\Lambda}_{k}C_{k}X_{2k}^{(0)}(\tau), \\
X_{2k}^{(1)}\bigg(\frac{\tau}{4\tau+1}\bigg) &= \frac{4}{k}e^{\frac{\pi i}{3}} A_{k}\Lambda_{k}^{-1}A_{k}X_{2k}^{(1)}(\tau).
\end{align}
Here $s_k,t_k,A_{k},\Lambda_{k},\widetilde{\Lambda}_{k},B_k$ and $C_k$ are defined in Theorems \ref{thm-G-o} and \ref{thm-G-e}.
As a consequence, both $X_{2k}^{(0)}(\tau)$ and $X_{2k}^{(1)}(\tau)$ are vector-valued modular functions on $\Gamma_0(4)$.
\end{corollary}

The rest of this paper is organized as follows. In Section \ref{sec-pre}, we review some basic facts about modular forms. In particular, we collect modular properties of some classical functions involving the Dedekind eta function, Weber's modular functions and  the theta series $g_{j,m}(\tau)$ and $h_{j,m}(\tau)$. We present proofs for Theorems \ref{thm-M-conj} and \ref{thm-M-Capparelli-conj} in Section \ref{sec-proof-M}.  In Section \ref{sec-theta}, we give modular transformation laws for $g_{j,m}(\tau)$ and $h_{j,m}(\tau)$, and then present proofs for Theorems \ref{thm-G-o} and \ref{thm-G-e} as well as Corollaries  \ref{cor-AG} and \ref{cor-B}. Finally, in Section \ref{sec-application} we apply our results on theta series to give more transformation formulas for some other Nahm sums in the literature.

\section{Preliminaries}  \label{sec-pre}

In this section, we review some basic knowledge about modular forms.

A useful tool to transform theta series to infinite products is the Jacobi triple product identity (see e.g.\ \cite[Theorem 1.3.3]{Berndt-book}):
\begin{equation}\label{JTP}
(z, q/z ,q; q)_\infty = \sum_{n=-\infty}^{\infty} (-1)^{n}q^{n(n-1)/2} z^{n}, \quad z\neq 0.
\end{equation}
For convenience, we use compact notations:
\begin{align}
    J_m:=(q^m;q^m)_\infty, \quad J_{a,m}:=(q^a,q^{m-a},q^m;q^m)_\infty.
\end{align}

We now recall some important functions which will be used frequently in establishing our transformation formulas.
\begin{enumerate}
\item The Dedekind eta function is defined as
\begin{align}
\eta(\tau):=q^{1/24}(q;q)_\infty.  \label{eta}
\end{align}
\item Weber's modular functions are given by \cite{We}
\begin{align}\label{Weber-defn}
\mathfrak{f}(\tau):=q^{-1/48} (-q^{1/2};q)_\infty, \ \  \mathfrak{f}_1(\tau):=q^{-1/48} (q^{1/2};q)_\infty, \ \ \mathfrak{f}_2(\tau):=q^{1/24}(-q;q)_\infty.
\end{align}

\item The generalized Dedekind eta function is defined by (see e.g.\ \cite{Yang})
\begin{equation}  \label{def-Geta}
E_{g,h}^{(N)}(\tau) := q^{\frac{1}{2}B(\frac{g}{N})} \prod_{m=1}^{\infty}\Big(1-\zeta_{N}^{h}q^{m-1+\frac{g}{N}}\Big)\Big(1-\zeta_{N}^{-h}q^{m-\frac{g}{N}}\Big),
\end{equation}
where $B(x) = x^2 -x + 1/6$, $N$ is a positive integer, $g$ and $h$ are arbitrary real numbers not simultaneously congruent to $0$ modulo $N$.
\end{enumerate}

The following properties about the Dedekind eta function and Weber's modular functions  are well-known:
\begin{align}
&\eta(-1/\tau)=\sqrt{-i \tau} \eta(\tau), \ \ \eta(\tau+1)=e^{\pi i  /12} \eta(\tau), \label{eta-tran} \\
&\mathfrak{f}(-1/\tau) = \mathfrak{f}(\tau), \ \ \ \mathfrak{f}_2(-1/\tau) = \frac{1}{\sqrt{2}}  \mathfrak{f}_1(\tau), \ \ \mathfrak{f}_1(-1/\tau)=\sqrt{2} \mathfrak{f}_2(\tau), \label{Weber-1} \\
&\mathfrak{f}(\tau+1) =e^{-\pi i /24}  \mathfrak{f}_1(\tau), \ \ \mathfrak{f}_1(\tau+1) = e^{-\pi i/24} \mathfrak{f}(\tau),  \ \ \mathfrak{f}_2(\tau+1) = e^{\pi i/12} \mathfrak{f}_2(\tau). \label{Weber-2}
\end{align}
It is  also easy to verify the following properties \cite[p.\ 215]{Wakimoto}:
\begin{align}
& h_{j,m}(\tau)=h_{-j,m}(\tau)=h_{2m+j,m}(\tau), \quad g_{j,m}(\tau)=g_{-j,m}(\tau)=-g_{2m+j,m}(\tau), \label{g-h-period}\\
& h_{j,m}(\tau)=h_{2j,4m}(\tau)+h_{4m-2j,4m}(\tau), \label{h-h-change} \\
& g_{j,m}(\tau)=h_{2j,4m}(\tau)-h_{4m-2j,4m}(\tau), \label{g-h-change}\\
& h_{j,m}(2\tau)=h_{2j,2m}(\tau), \quad g_{j,m}(2\tau)=g_{2j,2m}(\tau). \label{hg-double}
\end{align}

The generalized Dedekind eta function $E_{g,h}^{(N)}(\tau)$ satisfies the following transformation formulas.
\begin{lemma}(Cf. \cite[Theorem 1]{Yang}) \label{Geta-tran}
The function $E_{g,h}^{(N)}(\tau)$ satisfies
\begin{align}
E_{g+N,h}^{(N)}(\tau) = E_{-g,-h}^{(N)}(\tau) = -\zeta_{N}^{-h}E_{g,h}^{(N)}(\tau), \quad E_{g,h+N}^{(N)}(\tau) = E_{g,h}^{(N)}(\tau).
\end{align}
Moreover, for $\gamma = \begin{pmatrix}
\begin{smallmatrix}
a & b \\
c & d
\end{smallmatrix}
\end{pmatrix}
\in \mathrm{SL}(2,\mathbb{Z})$, we have
\begin{align}
E_{g,h}^{(N)}(\tau+b) = e^{\pi ibB(g/N)}E_{g,bg+h}^{(N)}(\tau), ~\text{for}~ c = 0,  \\
E_{g,h}^{(N)}(\gamma\tau) = \varepsilon(a,b,c,d)e^{\pi i \delta}E_{g',h'}^{(N)}(\tau), ~\text{for}~ c \neq 0,
\end{align}
where
\begin{align}
\varepsilon(a,b,c,d) = \begin{cases}
e^{\pi i(bd(1-c^2)+c(a+d-3))/6}, & ~ \text{if $c$ is odd}, \\
-ie^{\pi i(ac(1-d^2)+d(b-c+3))/6}, & ~\text{if $d$ is odd},
\end{cases}  \nonumber \\
\delta = \frac{g^2 ab+2ghbc+h^2 cd}{N^2} - \frac{gb+h(d-1)}{N},  \nonumber
\end{align}
and
\begin{equation}
(g', h')=(g,h)
\begin{pmatrix}
a & b \\
c & d
\end{pmatrix}.   \nonumber
\end{equation}
\end{lemma}

We recall transformation formulas for the theta series defined in \eqref{g-defn} and \eqref{hg}.
\begin{lemma}\label{lem-modular}
(Cf. \cite[p.\ 215, Theorem 4.5]{Wakimoto}.) For $j\in \mathbb{Z}$ and $m\in \frac{1}{2}\mathbb{N}$ we have
\begin{align}
h_{j,m}\left(-\frac{1}{\tau}\right)&=\frac{(-i\tau)^{\frac{1}{2}}}{\sqrt{2m}} \sum_{0\leq k\leq 2m-1} e^{\frac{\pi ijk}{m}}h_{k,m}(\tau), \\
g_{j,m}\left(-\frac{1}{\tau}\right)&=\frac{(-i\tau)^{\frac{1}{2}}}{\sqrt{2m}}\sum_{\begin{smallmatrix}
    0\leq k\leq 4m-1 \\ k ~~\text{odd}
\end{smallmatrix}} e^{\frac{\pi ijk}{2m}} h_{\frac{k}{2},m}(\tau).
\end{align}
For $j+m\in \mathbb{Z}$ we have
\begin{align}\label{hg-tran}
h_{j,m}(\tau+1)=e^{\frac{\pi ij^2}{2m}}h_{j,m}(\tau), \quad g_{j,m}(\tau+1)=e^{\frac{\pi ij^2}{2m}}g_{j,m}(\tau).
\end{align}
\end{lemma}

We also need the following simple fact.
\begin{lemma}\label{lem-action}
Let $N$ be a positive integer. Suppose that a vector-valued function $X(\tau)$ satisfies
\begin{align}\label{eq-lem-X-TS}
X(\tau+1)=\Lambda X(\tau), \quad X\left(-\frac{1}{N\tau}\right)=P X(\tau),
\end{align}
where $\Lambda$ and $P$ are matrices independent of $\tau$. Then we have
\begin{align}
X\left(\frac{\tau}{N\tau+1}\right)=P\Lambda^{-1} PX(\tau).
\end{align}
In particular, for $N=1,2,3,4$, if $X(\tau)$ satisfies \eqref{eq-lem-X-TS}, then $X(\tau)$ is a modular function on $\Gamma_0(N)$.
\end{lemma}
\begin{proof}
We have
\begin{align}\label{X-Gamma}
&X\left(\frac{\tau}{N\tau+1}\right)=X\left(-\frac{1}{-N(1+({N\tau})^{-1})}\right)=P X\left(-1-\frac{1}{N\tau}\right)\nonumber \\
&=P\Lambda^{-1}X\left(-\frac{1}{N\tau}\right)=P\Lambda^{-1}PX(\tau).
\end{align}
For $N=1,2,3,4$, note that the action of $\Gamma_0(N)$ on the upper half complex plane is generated by the action of the matrices $\left(\begin{smallmatrix} 1 & 1 \\ 0 & 1\end{smallmatrix}\right)$ and $\left(\begin{smallmatrix}
    1 & 0 \\ N & 1
\end{smallmatrix}\right)$. Therefore,  if $X(\tau)$ satisfies  \eqref{eq-lem-X-TS}, then $X(\tau)$ is modular on $\Gamma_0(N)$.
\end{proof}

\section{Proofs of Mizuno's formulas}\label{sec-proof-M}

In this section, we provide proofs for the formulas stated in Theorems \ref{thm-M-conj} and \ref{thm-M-Capparelli-conj}.

Before we present the proof of Theorem \ref{thm-M-conj}, we recall some 3-dissection formulas for the function
\begin{align}
F(z;q):=\frac{(q;q)_\infty}{(zq,q/z;q)_\infty}.
\end{align}
It is well-known that $F(z;q)$ is the generating function of the partition cranks. Garvan \cite[Eq.\ (2.15)]{Garvan} proved that for $z=\zeta_9^{j}$ where $j\in \{1,2,4,5,7,8\}$ we have
\begin{align}
F(z;q)&=\frac{J_{6,27}J_{12,27}}{J_{27}} -(1-z+z^2+z^5)q \frac{J_{3,27}J_{12,27}}{J_{27}} \nonumber \\
& \qquad \qquad +(z^2-z-z^4)q^2\frac{J_{3,27}J_{6,27}}{J_{27}}.  \label{F-3-dissection}
\end{align}
Note that Garvan stated \eqref{F-3-dissection} only for $z=\zeta_9$. However, we have checked that following the same arguments in \cite[Section 3]{Garvan}, it is easy to show that \eqref{F-3-dissection} actually holds for $z$ being any primitive 9-th root of unity.

\begin{proof}[Proof of Theorem \ref{thm-M-conj}]
The formula \eqref{X-KR-T} is trivial.

We denote the $i$-th component of $X(\tau)$ as $X_{i}(\tau) (i=1,2,3)$.
In terms of the generalized Dedekind eta functions \eqref{def-Geta}, we have
\begin{align}
X_1(\tau)= \frac{1}{E_{1,0}^{(9)}(9\tau)E_{3,0}^{(9)}(9\tau)}, ~
X_2(\tau) = \frac{1}{E_{2,0}^{(9)}(9\tau)E_{3,0}^{(9)}(9\tau)}, ~ X_3(\tau)  = \frac{1}{E_{3,0}^{(9)}(9\tau)E_{4,0}^{(9)}(9\tau)}.
\end{align}
Applying Lemma \ref{Geta-tran}, we deduce that
\begin{align}
X_1\Big(-\frac{1}{\tau}\Big) = & \frac{1}{E_{1,0}^{(9)}(-\frac{9}{\tau})E_{3,0}^{(9)}(-\frac{9}{\tau})} = e^{\frac{5}{9}\pi i}\frac{1}{E_{0,-1}^{(9)}(\frac{\tau}{9})E_{0,-3}^{(9)}(\frac{\tau}{9})}, \label{trantry-1} \\
X_2\Big(-\frac{1}{\tau}\Big) = & \frac{1}{E_{2,0}^{(9)}(-\frac{9}{\tau})E_{3,0}^{(9)}(-\frac{9}{\tau})} = e^{\frac{4}{9}\pi i}\frac{1}{E_{0,-2}^{(9)}(\frac{\tau}{9})E_{0,-3}^{(9)}(\frac{\tau}{9})},  \\
X_3\Big(-\frac{1}{\tau}\Big) = & \frac{1}{E_{3,0}^{(9)}(-\frac{9}{\tau})E_{4,0}^{(9)}(-\frac{9}{\tau})} = e^{\frac{2}{9}\pi i}\frac{1}{E_{0,-3}^{(9)}(\frac{\tau}{9})E_{0,-4}^{(9)}(\frac{\tau}{9})}.
\end{align}
By \eqref{trantry-1} we have
\begin{align}
&X_1\left(-\frac{1}{9\tau}\right)=e^{5\pi i/9}q^{-1/6}\times \frac{1}{(\zeta_9^{-1},\zeta_9q,\zeta_3^{-1},\zeta_3q;q)_\infty} \nonumber \\
&=q^{-1/6}\frac{\zeta_{18}^5}{(1-\zeta_9^{-1})(1-\zeta_3^{-1})} \times \frac{(q;q)_\infty}{(\zeta_9^{-1}q,\zeta_9q;q)_\infty (\zeta_3^{-1}q,\zeta_3q,q;q)_\infty} \nonumber \\
&=q^{-1/6}\frac{\zeta_{18}^5}{(1-\zeta_9^{-1})(1-\zeta_3^{-1})(q^3;q^3)_\infty}F(\zeta_9;q) \nonumber \\
&=a_{11}X_1(3\tau)+a_{12}X_2(3\tau)+a_{13}X_3(3\tau), \quad \text{(by \eqref{F-3-dissection})}
\end{align}
where
\begin{equation}\label{a1-value}
\begin{split}
a_{11}=\frac{\zeta_{18}^5}{(1-\zeta_9^{-1})(1-\zeta_3^{-1})}=\frac{1}{2\sqrt{3}\sin \frac{\pi}{9}}=\alpha_1, \\
a_{12}=-\frac{\zeta_{18}^5(1-\zeta_9+\zeta_9^2+\zeta_9^5)}{(1-\zeta_9^{-1})(1-\zeta_3^{-1})}=\frac{1}{2\sqrt{3}\sin \frac{2\pi}{9}}=\alpha_2,\\
a_{13}=\frac{\zeta_{18}^5(\zeta_9^2-\zeta_9-\zeta_9^4)}{(1-\zeta_9^{-1})(1-\zeta_3^{-1})}=\frac{1}{2\sqrt{3}\sin \frac{4\pi}{9}}=\alpha_4.
\end{split}
\end{equation}
Similarly,
\begin{align}
&X_2\left(-\frac{1}{9\tau}\right)=e^{4\pi i/9}q^{-1/6}\times \frac{1}{(\zeta_9^{-2},\zeta_9^2q,\zeta_3^{-1},\zeta_3q;q)_\infty} \nonumber \\
&=\frac{\zeta_{9}^2}{(1-\zeta_9^{-2})(1-\zeta_3^{-1})(q^3;q^3)_\infty}q^{-1/6}F(\zeta_9^2;q) \nonumber \\
&=a_{21}X_1(3\tau)+a_{22}X_2(3\tau)+a_{23}X_3(3\tau), \quad \text{(by \eqref{F-3-dissection})}
\end{align}
where
\begin{equation}\label{a2-value}
\begin{split}
a_{21}=\frac{\zeta_{9}^2}{(1-\zeta_9^{-2})(1-\zeta_3^{-1})}=\frac{1}{2\sqrt{3}\sin \frac{2\pi}{9}}=\alpha_2,  \\
a_{22}=-\frac{\zeta_{9}^2(1-\zeta_9^2+\zeta_9^4+\zeta_9^{10})}{(1-\zeta_9^{-2})(1-\zeta_3^{-1})}=-\frac{1}{2\sqrt{3}\sin \frac{4\pi}{9}}=-\alpha_4,\\
a_{23}=\frac{\zeta_{9}^2(\zeta_9^4-\zeta_9^2-\zeta_9^8)}{(1-\zeta_9^{-2})(1-\zeta_3^{-1})}=-\frac{1}{2\sqrt{3}\sin \frac{\pi}{9}}=-\alpha_1.
\end{split}
\end{equation}
In the same way,
\begin{align}
&X_3\left(-\frac{1}{9\tau}\right)=e^{2\pi i/9}q^{-1/6}\times \frac{1}{(\zeta_9^{-4},\zeta_9^4q,\zeta_3^{-1},\zeta_3q;q)_\infty}  \nonumber \\
&=q^{-1/6}\frac{\zeta_9}{(1-\zeta_9^{-4})(1-\zeta_3^{-1})(q^3;q^3)_\infty}F(\zeta_9^4;q) \nonumber \\
&=a_{31}X_1(3\tau)+a_{32}X_2(3\tau)+a_{33}X_3(3\tau), \quad \text{(by \eqref{F-3-dissection})}
\end{align}
where
\begin{equation}\label{a3-value}
\begin{split}
a_{31}=\frac{\zeta_9}{(1-\zeta_9^{-4})(1-\zeta_3^{-1})}=\frac{1}{2\sqrt{3}\sin \frac{4\pi}{9}}=\alpha_4, \\
a_{32}=-\frac{\zeta_9(1-\zeta_9^4+\zeta_9^8+\zeta_9^{20})}{(1-\zeta_9^{-4})(1-\zeta_3^{-1})}=-\frac{1}{2\sqrt{3}\sin \frac{\pi}{9}}=-\alpha_1,\\
a_{33}=\frac{\zeta_9(\zeta_9^8-\zeta_9^4-\zeta_9^{16})}{(1-\zeta_9^{-4})(1-\zeta_3^{-1})}=\frac{1}{2\sqrt{3}\sin \frac{2\pi}{9}}=\alpha_2.
\end{split}
\end{equation}
This proves \eqref{goal}.

Now by Lemma \ref{lem-action} we know that $X(\tau)$ is a modular function on $\Gamma_0(3)$.
\end{proof}

\begin{rem}
The last second equalities in each of the equations in \eqref{a1-value}--\eqref{a3-value} can be achieved by direct calculations. Alternatively, we give a way to verify it. Take the first equation in \eqref{a1-value} as an example. Note that
$$e^{k \pi i/n}=\cos \frac{k\pi}{n}+i\sin \frac{k\pi }{n}, \quad \zeta_9^3-\zeta_9^{-3}=\sqrt{3}i.$$
We see that the last second equality in the first equation in \eqref{a1-value} is equivalent to
\begin{align}
\frac{\zeta_{18}^{5}}{(1-\zeta_9^{-1})(1-\zeta_9^{-3})}=-\frac{1}{(\zeta_9^3-\zeta_9^{-3})(\zeta_9^4-\zeta_9^{-4})}.
\end{align}
Canceling common factors, we see that this is equivalent to
\begin{align}
\zeta_{18}^{11}(1-\zeta_9^8)=(1-\zeta_9), \quad \text{i.e.} \quad \zeta_{18}^9=-1.
\end{align}
The last equality is trivial.
\end{rem}
\begin{rem}
After this paper has been written,  we heard from Yuma Mizuno that he \cite{Mizuno-note}  has independently found another proof for Theorem \ref{thm-M-conj}. His proof is based on the Macdonald identity of type $A_2$ and some transformation formulas of binary theta series.
\end{rem}

\begin{proof}[Proof of Theorem \ref{thm-M-Capparelli-conj}]
The formula \eqref{X-Capparelli-T} is trivial.

We denote the $i$-th component of $X(\tau)$ by $X_i(\tau)$ ($i=1,2$). From \eqref{eq-Capparelli-1} and \eqref{eq-Capparelli-2} we have
\begin{align}
X_1(\tau)=\frac{\eta(4\tau)\eta^2(6\tau)}{\eta(2\tau)\eta(3\tau)\eta(12\tau)}, \quad X_2(\tau)=\frac{\eta^2(2\tau)\eta(12\tau)}{\eta(\tau)\eta(4\tau)\eta(6\tau)}.
\end{align}
Using \eqref{eta-tran} we deduce that
\begin{align}\label{g1-S}
X_1\left(-\frac{1}{\tau}\right)&=\frac{\eta(-4/\tau)\eta^2(-6/\tau)}{\eta(-2/\tau)\eta(-3/\tau)\eta(-12/\tau)}
=\frac{1}{\sqrt{2}}\frac{\eta(\tau/4)\eta^2(\tau/6)}{\eta(\tau/2)\eta(\tau/3)\eta(\tau/12)}.
\end{align}
Similarly,
\begin{align}\label{g2-S}
X_2\left(-\frac{1}{\tau}\right)&=\frac{\eta^2(-2/\tau)\eta(-12/\tau)}{\eta(-1/\tau)\eta(-4/\tau)\eta(-6/\tau)}
=\frac{1}{\sqrt{2}}\frac{\eta^2(\tau/2)\eta(\tau/12)}{\eta(\tau)\eta(\tau/4)\eta(\tau/6)}.
\end{align}
Using the method in \cite{Garvan-Liang}, it is easy to verify that
\begin{align}
\frac{\eta(\tau/4)\eta^2(\tau/6)}{\eta(\tau/2)\eta(\tau/3)\eta(\tau/12)}=\frac{\eta(4\tau/3)\eta^2(2\tau)}{\eta(2\tau/3)\eta(\tau)\eta(4\tau)}
+\frac{\eta^2(2\tau/3)\eta(4\tau)}{\eta(\tau/3)\eta(4\tau/3)\eta(2\tau)}, \label{id-1} \\
\frac{\eta^2(\tau/2)\eta(\tau/12)}{\eta(\tau)\eta(\tau/4)\eta(\tau/6)}=\frac{\eta(4\tau/3)\eta^2(2\tau)}{\eta(2\tau/3)\eta(\tau)\eta(4\tau)}
-\frac{\eta^2(2\tau/3)\eta(4\tau)}{\eta(\tau/3)\eta(4\tau/3)\eta(2\tau)}. \label{id-2}
\end{align}
This proves \eqref{X-Capparelli-S}.

Now by Lemma \ref{lem-action} we know that $X(\tau)$ is a modular function on $\Gamma_0(3)$.
\end{proof}
\begin{rem}
The identities \eqref{id-1} and \eqref{id-2} can also be stated as
\begin{align}
\frac{J_3}{J_1}=\frac{J_4J_6J_{16}J_{24}^2}{J_2^2J_8J_{12}J_{48}}+q\frac{J_4J_6J_8^2J_{48}}{J_2^2J_4J_{16}J_{24}}, \label{J-id-1} \\
\frac{J_1}{J_3}=\frac{J_2J_{12}J_{16}J_{24}^2}{J_6^2J_8J_{12}J_{48}}-q\frac{J_2J_{12}J_8^2J_{48}}{J_4J_6^2J_{16}J_{24}}. \label{J-id-2}
\end{align}
This gives 2-dissection formulas for the infinite products in the left sides.
\end{rem}

\section{Proofs of Theorems \ref{thm-G-o}--\ref{thm-G-e} and Corollaries \ref{cor-AG}--\ref{cor-B}}   \label{sec-theta}

We first establish some transformation formulas for the theta functions $g_{j,k}(\tau)$ and $h_{j,k}(\tau)$ defined in \eqref{g-defn} and \eqref{hg}.
\begin{lemma}  \label{lem-g-new-tran}
For $k\geq 2$ and $0\leq j\leq k$, we have
\begin{equation}
g_{j,k}\big(-\frac{1}{4\tau}\big) =
\begin{footnotesize}
\begin{cases}
\frac{2(-2i\tau)^{\frac{1}{2}}}{\sqrt{k}} \sum\limits_{0\leq l\leq \frac{k-2}{2}}\cos \frac{(2l+1)j\pi}{2k}h_{2l+1,k}(\tau) & \text{k even, j even}, \\
\frac{2(-2i\tau)^{\frac{1}{2}}}{\sqrt{k}} \sum\limits_{0\leq l\leq \frac{k-2}{2}}\cos \frac{(2l+1)j\pi}{2k}g_{2l+1,k}(\tau) & \text{k even, j odd},  \\
\frac{2(-2i\tau)^{\frac{1}{2}}}{\sqrt{k}} \Big(\sum\limits_{0\leq l\leq \frac{k-3}{2}}\cos \frac{(2l+1)j\pi}{2k}h_{2l+1,k}(\tau) + \frac{1}{2}\cos \frac{j\pi}{2}h_{k,k}(\tau) \Big)  & \text{k odd, j even}, \\
\frac{2(-2i\tau)^{\frac{1}{2}}}{\sqrt{k}} \sum\limits_{0\leq l\leq \frac{k-3}{2}}\cos \frac{(2l+1)j\pi}{2k}g_{2l+1,k}(\tau) & \text{k odd, j odd}.
\end{cases}
\end{footnotesize}
\end{equation}
\end{lemma}

\begin{proof}
From Lemma \ref{lem-modular} and \eqref{g-h-period}, \eqref{g-h-change}, we have that
\begin{align}
& g_{j,k}\left(-\frac{1}{4\tau}\right) =  \frac{(-2i\tau)^{\frac{1}{2}}}{\sqrt{k}} \sum_{\substack{0 \leq l \leq 4k-1 \\ l~\text{odd} }} e^{\frac{\pi i l j}{2k}}h_{\frac{l}{2},k}(4\tau)  \nonumber \\
& = \frac{(-2i\tau)^{\frac{1}{2}}}{\sqrt{k}} \sum_{0\leq l \leq 2k-1} e^{\frac{2l+1}{2k}\pi ij}h_{\frac{2l+1}{2},k}(4\tau) \nonumber \\
& = \frac{(-2i\tau)^{\frac{1}{2}}}{\sqrt{k}} \sum_{0\leq l \leq k-1} \Big( e^{\frac{2l+1}{2k}\pi ij}h_{\frac{2l+1}{2},k}(4\tau) + e^{\frac{4k-(2l+1)}{2k}\pi i j}h_{\frac{4k-(2l+1)}{2},k}(4\tau)\Big) \nonumber \\
& = 2\frac{(-2i\tau)^{\frac{1}{2}}}{\sqrt{k}}\sum_{0\leq l \leq k-1} \cos \frac{(2l+1)j\pi}{2k} h_{\frac{2l+1}{2},k}(4\tau).   \label{A-4}
\end{align}
Note that the parity of $k$ and $t$ will affect the subsequent calculations. So we split our proof into four cases. \\
\textbf{Case 1:} If $k$ is even, then by \eqref{A-4} and \eqref{g-h-change},
\begin{align*}
& g_{j,k}\left(-\frac{1}{4\tau}\right) = \frac{2(-2i\tau)^{\frac{1}{2}}}{\sqrt{k}}  \bigg(\sum_{0\leq l \leq \frac{k-2}{2}} \Big(\cos \frac{(2l+1)j\pi}{2k}h_{4l+2,4k}(\tau) \nonumber \\
& \quad + \cos \frac{(2k-1-2l)j\pi}{2k}h_{4k-4l-2,4k}(\tau) \Big) \bigg)  \nonumber \\
& = \frac{2(-2i\tau)^{\frac{1}{2}}}{\sqrt{k}}\Big(\sum_{0\leq l \leq \frac{k-2}{2}} \cos \frac{(2l+1)j\pi}{2k}\big(h_{4l+2,4k}(\tau) +(-1)^j h_{4k-4l-2,4k}(\tau) \big) \Big).
\end{align*}
In particular, when $j$ is even (resp.\ odd), we obtain the first (resp.\ second) case using \eqref{h-h-change} (resp.\ \eqref{g-h-change}).

\textbf{Case 2:} If $k$ is odd, then by \eqref{A-4} and \eqref{h-h-change},
\begin{align*}
& g_{j,k}\left(-\frac{1}{4\tau}\right) = \frac{2(-2i\tau)^{\frac{1}{2}}}{\sqrt{k}}  \bigg(\sum_{0\leq l \leq \frac{k-3}{2}} \Big(\cos \frac{(2l+1)j\pi}{2k}h_{4l+2,4k}(\tau)  \nonumber \\
& \quad +  \cos \frac{(2k-1-2l)j\pi}{2k}h_{4k-4l-2,4k}(\tau) \Big) + \cos \frac{j\pi}{2}h_{2k,4k}(\tau) \bigg)  \nonumber \\
& = \frac{2(-2i\tau)^{\frac{1}{2}}}{\sqrt{k}}\Big(\sum_{0\leq l \leq \frac{k-3}{2}} \cos \frac{(2l+1)j\pi}{2k}\big(h_{4l+2,4k}(\tau) + (-1)^j h_{4k-4l-2,4k}(\tau) \big)  \nonumber \\
&  \quad + \frac{1}{2}\cos \frac{j\pi}{2}h_{k,k}(\tau) \Big).
\end{align*}
In particular, when $j$ is even (resp.\ odd), we obtain the third (resp.\ fourth) case using \eqref{h-h-change} (resp.\ \eqref{g-h-change}).
\end{proof}

\begin{lemma}\label{lem-h-new-tran}
For $k\geq 2$ and $0\leq j\leq k$, we have
\begin{equation}
h_{j,k}\big(-\frac{1}{4\tau}\big) =
\begin{footnotesize}
\begin{cases}
\frac{2(-2i\tau)^{\frac{1}{2}}}{\sqrt{k}} \Big(\sum\limits_{1\leq l\leq \frac{k-2}{2}}\cos \frac{lj\pi}{k}h_{2l,k}(\tau) + \frac{1}{2}h_{0,k}(\tau)+\frac{1}{2}\cos \frac{j\pi}{2}h_{k,k}(\tau)\Big) & \text{k even, j even}, \\
\frac{2(-2i\tau)^{\frac{1}{2}}}{\sqrt{k}} \Big(\sum\limits_{1\leq l\leq \frac{k-2}{2}}\cos \frac{lj\pi}{k}g_{2l,k}(\tau)+\frac{1}{2}g_{0,k}(\tau)\Big) & \text{k even, j odd},  \\
\frac{2(-2i\tau)^{\frac{1}{2}}}{\sqrt{k}} \Big(\sum\limits_{1\leq l\leq \frac{k-1}{2}}\cos \frac{lj\pi}{k}h_{2l,k}(\tau) + \frac{1}{2}h_{0,k}(\tau) \Big)  & \text{k odd, j even}, \\
\frac{2(-2i\tau)^{\frac{1}{2}}}{\sqrt{k}} \Big(\sum\limits_{1\leq l\leq \frac{k-1}{2}}\cos \frac{lj\pi}{k}g_{2l,k}(\tau)+\frac{1}{2}g_{0,k}(\tau)\Big) & \text{k odd, j odd}.
\end{cases}
\end{footnotesize}
\end{equation}
\end{lemma}
\begin{proof}
From Lemma \ref{lem-modular} and \eqref{g-h-period}, \eqref{hg-double}, we have that
\begin{align}
& h_{j,k}\left(-\frac{1}{4\tau}\right) = \frac{(-2i\tau)^{\frac{1}{2}}}{\sqrt{k}}\sum_{0\leq l\leq 2k-1}e^{\frac{\pi ijl}{k}}h_{l,k}(4\tau) \nonumber \\
& = \frac{(-2i\tau)^{\frac{1}{2}}}{\sqrt{k}} \sum_{0\leq l\leq k-1}\Big(e^{\frac{\pi ijl}{k}}h_{l,k}(4\tau)+e^{\frac{2k-1-l}{k}\pi ij}h_{2k-1-l,k}(4\tau)\Big) \nonumber \\
& = \frac{(-2i\tau)^{\frac{1}{2}}}{\sqrt{k}}\Big(2\sum_{1\leq l\leq k-1}\cos \frac{lj\pi}{k}h_{l,k}(4\tau)+h_{0,k}(4\tau)+(-1)^{j} h_{k,k}(4\tau)\Big).  \label{h-4-1}
\end{align}
Similar to the proof of Lemma \ref{lem-g-new-tran}, we split our proof into four cases.\\
\textbf{Case 1:} If $k$ is even, then by \eqref{h-4-1} and \eqref{g-h-period},
\begin{align*}
& h_{j,k}\left(-\frac{1}{4\tau}\right) = \frac{2(-2i\tau)^{\frac{1}{2}}}{\sqrt{k}}\Big( \sum_{1\leq l\leq \frac{k-2}{2}}\big( \cos\frac{lj\pi}{k}h_{4l,4k}(\tau)+\cos\frac{(k-l)j\pi}{k}h_{4k-4l,4k}(\tau) \big) \nonumber \\
& \quad +\cos \frac{j\pi}{2}h_{2k,4k}(\tau) +\frac{1}{2}\big( h_{0,4k}(\tau)+(-1)^{j} h_{4k,4k}(\tau)\big)\Big) \nonumber \\
& = \frac{2(-2i\tau)^{\frac{1}{2}}}{\sqrt{k}}\Big( \sum_{1\leq l\leq \frac{k-2}{2}}\cos \frac{lj\pi}{k}\big( h_{4l,4k}(\tau)+(-1)^{j}h_{4k-4l,4k}(\tau) \big) +\cos \frac{j\pi}{2}h_{2k,4k}(\tau) \nonumber \\
& \quad + \frac{1}{2}\big(  h_{0,4k}(\tau)+(-1)^{j} h_{4k,4k}(\tau)\big) \Big).
\end{align*}
In particular, when $j$ is even (resp. odd), we deduce the first (resp. second) case using \ref{h-h-change} (resp. \eqref{g-h-change}). \\
\textbf{Case 2:} If $k$ is odd, then by \eqref{h-4-1} and \eqref{g-h-period},
\begin{align*}
& h_{j,k}\left(-\frac{1}{4\tau}\right) = \frac{2(-2i\tau)^{\frac{1}{2}}}{\sqrt{k}}\Big( \sum_{1\leq l\leq \frac{k-1}{2}}\big( \cos\frac{lj\pi}{k}h_{4l,4k}(\tau)+\cos\frac{(k-l)j\pi}{k}h_{4k-4l,4k}(\tau) \big) \nonumber \\
& \quad +\frac{1}{2}\big( h_{0,4k}(\tau)+(-1)^{j} h_{4k,4k}(\tau)\big)\Big) \nonumber \\
& = \frac{2(-2i\tau)^{\frac{1}{2}}}{\sqrt{k}}\Big( \sum_{1\leq l\leq \frac{k-1}{2}}\cos \frac{lj\pi}{k}\big( h_{4l,4k}(\tau)+(-1)^{j}h_{4k-4l,4k}(\tau) \big) \nonumber \\
& \quad + \frac{1}{2}\big(  h_{0,4k}(\tau)+(-1)^{j} h_{4k,4k}(\tau)\big) \Big).
\end{align*}
In particular, when $j$ is even (resp. odd), we deduce the third (resp. fourth) case using \ref{h-h-change} (resp. \eqref{g-h-change}).
\end{proof}

Now we are ready to prove Theorems \ref{thm-G-o}--\ref{thm-G-e} and their corollaries.

\begin{proof}[Proof of Theorem \ref{thm-G-o}]
The first formula in \eqref{Gk-odd-tran} is trivial, and the second formula follows from Lemma \ref{lem-g-new-tran}. From \eqref{Gk-odd-tran} we have
\begin{align*}
&G_{k}^{(1)}\left(\frac{\tau}{4\tau+1}\right)  = \frac{2}{\sqrt{k}} \times \left(\frac{4\tau+1}{2\tau}i\right)^{\frac{1}{2}}
A_k G_{k}^{(1)}\left(-1-\frac{1}{4\tau}\right) \\
& = \frac{2}{\sqrt{k}} \times \left(\frac{4\tau+1}{2\tau}i\right)^{\frac{1}{2}}
A_k \Lambda_k^{-1}G_{k}^{(1)}\left(-\frac{1}{4\tau}\right) \\
 & =  \frac{4\sqrt{4\tau+1}}{k}A_k\Lambda_k^{-1}A_kG_k^{(1)}(\tau). \qedhere
\end{align*}
\end{proof}

\begin{proof}[Proof of Theorem \ref{thm-G-e}]
The formula \eqref{G_k-1} follows directly from \eqref{hg-tran}.
Note that from Lemma \ref{lem-g-new-tran}, we have
\begin{align}
G_{k}^{(0)}\left(\frac{\tau}{4\tau+1}\right) & = \frac{2}{\sqrt{k}} \times \left(\frac{4\tau+1}{2\tau}i\right)^{\frac{1}{2}}
B_k H_{k}^{(1)}\left(-1-\frac{1}{4\tau}\right), \label{pf-1}
\end{align}
where $H_{k}^{(1)}(\tau):= \big(h_{1,k}(\tau), h_{3,k}(\tau),\dots, h_{\widetilde{t}_{k},k}(\tau)\big)^{\mathrm{T}}$.
Next, by \eqref{hg-tran} and Lemma \ref{lem-h-new-tran}, we deduce that
\begin{align}
& H_{k}^{(1)}\left(-1-\frac{1}{4\tau}\right) = \diag(e^{-\frac{1}{2k}\pi i}, e^{-\frac{9}{2k}\pi i}, \dots, e^{-\frac{\widetilde{t}_{k}^2}{2k}\pi i}) H_{k}^{(1)}\left(-\frac{1}{4\tau}\right) \nonumber \\
& =  \frac{2(-2i\tau)^{\frac{1}{2}}}{\sqrt{k}}\widetilde{\Lambda}_{k}
C_k
G_{k}^{(0)}(\tau). \label{pf-2}
\end{align}
Substituting \eqref{pf-2} into \eqref{pf-1}, we obtain \eqref{G_k-2}.
\end{proof}

\begin{proof}[Proof of Corollary \ref{cor-AG}]
The formula \eqref{AG-T} follows from \eqref{eta-tran} and \eqref{hg-tran}. To prove \eqref{AG-S}, note that
\begin{align}
X_{2k+1}(\tau)=\frac{1}{\eta(\tau)}\Big(g_{1,2k+1}\left(\frac{\tau}{2}\right),\dots,g_{2k-1,2k+1}\left(\frac{\tau}{2}\right)\Big)^\mathrm{T}=\frac{1}{\eta(\tau)}G_{2k+1}^{(1)}\left(\frac{\tau}{2}\right).
\end{align}
Therefore, by Theorem  \ref{thm-G-o} we deduce that
\begin{align}
&X_{2k+1}\left(-\frac{1}{\tau} \right)=\frac{1}{\eta(-1/\tau)}G_{2k+1}^{(1)}\left(-\frac{1}{2\tau}\right)\nonumber \\
&=\frac{1}{\sqrt{-i\tau}\eta(\tau)}\cdot \frac{2\sqrt{-i\tau}}{\sqrt{2k+1}}A_{2k+1}G_{2k+1}^{(1)}\left(\frac{\tau}{2}\right)=\frac{2}{\sqrt{2k+1}}A_{2k+1}X_{2k+1}(\tau). \qedhere
\end{align}
\end{proof}

\begin{proof}[Proof of Corollary \ref{cor-B}]
Note that
\begin{equation}
X_{2k}^{(0)}(\tau) = \frac{1}{\eta(\tau)}G_k^{(0)}(\tau), \quad X_{2k}^{(1)}(\tau) = \frac{1}{\eta(\tau)}G_k^{(1)}(\tau).
\end{equation}
The desired assertions follow from Theorems \ref{thm-G-o}--\ref{thm-G-e} and  \eqref{eta-tran}.
\end{proof}

\section{Applications to other Nahm sums}\label{sec-application}
A number of Nahm sums in the literature can be expressed as the infinite products  $g_{j,k}(\tau)/\eta(\tau)$.
See \cite[Eqs. (3.2), (3.6)]{WWg2}, \cite[Theorems 3.2, 3.4, 3.12--3.15 and Eqs.\ (3.77)--(3.80)]{WW112}, \cite[Theorems 3.13, 3.14, 3.17, 3.18, 3.25--3.27 and Eqs. (3.169), (3.171)]{WW122}, \cite[Theorems 3.4 and 3.6]{Wang2022}  and \cite[Eqs.\ (4.39), (4.40), (4.74), (4.76), Theorems 4.6, 4.11]{W-rank3} for example. Using Corollaries \ref{cor-AG} and \ref{cor-B}, we can deduce transformation formulas for these Nahm sums for free.

Meanwhile, there are some Nahm sums which can be expressed as $F(\tau)g_{j,k}(\tau)$ where $F(\tau)$ is some eta quotients, and we can obtain their transformation formulas combining the properties of $F(\tau)$ and Theorems \ref{thm-G-o} and \ref{thm-G-e}. In the examples below, the transformation properties of $F(\tau)$ can be obtained using  properties of the Weber functions and the Dedekind eta function listed in Section \ref{sec-pre}. Since the proofs are similar to Corollaries \ref{cor-AG} and \ref{cor-B}, we omit the details.
\begin{exam}
Recall the Nahm sums fom \cite[Theorem 4.1]{WW112}: For $k\geq 2, 1\leq i\leq k+1$, we have
\begin{align}
&x_{2k+3,i}(q):=\sum_{n_1,\dots,n_k\geq 0}\frac{q^{N_{1}^{2}+\cdots+N_{k}^{2}+2N_i+2N_{i+2}+2N_{i+4}+\cdots}}{(q^2;q^2)_{n_1}\cdots(q^2;q^2)_{n_{k-1}}(q^4;q^4)_{n_{k}}} \nonumber \\
&= \frac{(-q;q^2)_\infty (q^i,q^{2k+3-i},q^{2k+3};q^{2k+3})_\infty}{(q^2;q^2)_\infty},  \label{AG-ge-2k+3}
\end{align}
where $N_i=n_i+n_{i+1}+\cdots+n_k$ for $1\leq i\leq k$ and $N_{k+1}=0$. In addition, for $k=1$ we define $x_{5,i}$ ($i=1,2$) as the infinite product on the right side. Let
\begin{align}\label{exam1-X-defn}
X_{2k+3}(\tau) &:= \Big(q^{\frac{1}{8(2k+3)}-\frac{1}{8}}x_{2k+3,k+1}(q), \dots, q^{\frac{(2k+1)^2}{8(2k+3)}-\frac{1}{8}}x_{2k+3,1}(q)\Big)^{\mathrm{T}} \nonumber \\
&=\frac{\mathfrak{f}(2\tau)}{\eta(2\tau)}G_{2k+3}^{(1)}\Big(\frac{\tau}{2}\Big).
\end{align}
We have
\begin{align}
X_{2k+3}(\tau+1) &= \diag (e^{(\frac{1}{4(2k+3)}-\frac{1}{4})\pi i}, \dots, e^{(\frac{(2k+1)^2}{4(2k+3)}-\frac{1}{4})\pi i})X_{2k+3}(\tau), \label{exam1-X-T} \\
X_{2k+3}\Big(\frac{\tau}{4\tau+1}\Big) &= \frac{4}{2k+3} e^{\frac{\pi i}{4}}A_{2k+3}\Lambda_{2k+3}^{-1}A_{2k+3}X_{2k+3}(\tau). \label{exam1-X-S}
\end{align}
Here $A_{2k+3}, \Lambda_{2k+3}$ are defined in Theorem \ref{thm-G-o}.
As a consequence, $X_{2k+3}(\tau)$ is a vector-valued modular function on $\Gamma_0(4)$.
\end{exam}

Several Nahm sums in the literature (see e.g.\ \cite[Eqs. (3.25)--(3.27)]{WWg2}, \cite[Theorems 3.16--3.18]{WW112} and \cite[Theorems 3.11--3.12]{WW122}) can be expressed as the infinite products in \eqref{AG-ge-2k+3}.

\begin{exam}
Recall the generalized Nahm sums from  \cite[Theorem 4.5]{WW112}:
\begin{align}
&x_{2k+3,i}(q) := \sum_{m_1,m_2,n_1,\dots,n_k\geq 0}\frac{q^{\tbinom{m_1+1}{2} +m_1m_2 + 2\tbinom{m_2+1}{2} +2m_2n_1+4N_1^2+\cdots+4N_k^2+4N_i+\cdots+4N_k}}{(q;q)_{m_1}(q;q)_{m_2}(q^2;q^2)_{n_1}(q^4;q^4)_{n_2}\cdots (q^4;q^4)_{n_k}}  \nonumber \\
& = \frac{(q^{4i},q^{8k+12-4i},q^{8k+12};q^{8k+12})_\infty}{(q;q)_\infty}, \quad k\geq 1, 1\leq i\leq k+1,
\end{align}
where $N_i=n_i+n_{i+1}+\cdots+n_k$ for $1\leq i\leq k$ and $N_{k+1}=0$.
Let
\begin{align}
X_{2k+3}(\tau) & := \Big( q^{\frac{1}{2(2k+3)}-\frac{1}{24}}x_{2k+3,k+1}(q), \dots, q^{\frac{(2k+1)^2}{2(2k+3)}-\frac{1}{24}}x_{2k+3,1}(q) \Big)^{\mathrm{T}} \nonumber \\
& = \frac{1}{\eta(\tau)} G_{2k+3}^{(1)}(2\tau).
\end{align}
We have
\begin{align}
X_{2k+3}(\tau+1) = \diag (e^{(\frac{1}{2k+3}-\frac{1}{12})\pi i}, \dots, e^{(\frac{(2k+1)^2}{2k+3}-\frac{1}{12})\pi i})X_{2k+3}(\tau),  \\
X_{2k+3}\Big(\frac{\tau}{4\tau+1}\Big) = \frac{4}{2k+3}e^{\frac{\pi i}{3}}A_{2k+3}\widehat{\Lambda}_{2k+3}A_{2k+3}X_{2k+3}(\tau).
\end{align}
Here $\widehat{\Lambda}_{2k+3} = \diag (e^{-\frac{1}{4(2k+3)}\pi i}, e^{-\frac{9}{4(2k+3)}\pi i}, \dots, e^{-\frac{(2k+1)^2}{4(2k+3)}\pi i})$ and $A_{2k+3}$ is given in Theorem \ref{thm-G-o}.
As a consequence, $X_{2k+3}(\tau)$ is a vector-valued modular function on $\Gamma_0(4)$.

In particular, the case $k=1$ yields modular transformation formulas for the Nahm sums in \cite[Theorems 3.1 and 3.3]{WW112}.
\end{exam}

\begin{exam}
Recall the following Nahm sums in \cite[Theorem 4.7]{W-rank3}:
\begin{align}
x_1(q) := &  \sum_{i,j,k\geq 0}\frac{q^{3i^2 +2j^2+k^2 +4ij+2ik+2jk}}{(q^2;q^2)_i (q^2;q^2)_j (q^2;q^2)_k} = \frac{(-q;q^2)_\infty}{(q^2,q^8;q^{10})_\infty}, \label{x1-111-mod5}  \\
x_2(q) := &  \sum_{i,j,k\geq 0}\frac{q^{3i^2 +2j^2+k^2 +4ij+2ik+2jk+2i+2j}}{(q^2;q^2)_i (q^2;q^2)_j (q^2;q^2)_k} = \frac{(-q;q^2)_\infty}{(q^4,q^6;q^{10})_\infty},  \label{x3-111-mod5} \\
x_3(q) := &  \sum_{i,j,k\geq 0}\frac{q^{3i^2 +2j^2+k^2 +4ij+2ik+2jk+i+k}}{(q^2;q^2)_i (q^2;q^2)_j (q^2;q^2)_k} = \frac{(-q^2;q^2)_\infty}{(q^2,q^8;q^{10})_\infty}, \label{x2-111-mod5} \\
x_4(q) := & \sum_{i,j,k\geq 0}\frac{q^{3i^2 +2j^2+k^2 +4ij+2ik+2jk+3i+2j+k}}{(q^2;q^2)_i (q^2;q^2)_j (q^2;q^2)_k} = \frac{(-q^2;q^2)_\infty}{(q^4,q^6;q^{10})_\infty}. \label{x4-111-mod5}
\end{align}
Let
\begin{align}
&X(\tau) := (q^{-3/40}x_1(q), q^{13/40}x_2(q), q^{1/20}x_3(q), q^{9/20}x_4(q))^{\mathrm{T}} \nonumber \\
&=\frac{1}{\eta(2\tau)}(\mathfrak{f}(2\tau)g_{1,5}(\tau),
\mathfrak{f}(2\tau)g_{3,5}(\tau), \mathfrak{f}_2(2\tau)
g_{1,5}(\tau),\mathfrak{f}_2(2\tau)
g_{3,5}(\tau))^\mathrm{T}.
\end{align}
We have
\begin{align}
X(\tau+1)& = \diag{(\zeta_{40}^{-3}, \zeta_{40}^{13}, \zeta_{20}, \zeta_{20}^{9})}X(\tau),  \label{B-1} \\
X\left(\frac{\tau}{2\tau+1}\right)& = \frac{4}{5} M P N X(\tau). \label{X-111-mod5-S}
\end{align}
Here $P = \diag{(\zeta_{80}^{3}, \zeta_{80}^{-13}, \zeta_{80}^{3}, \zeta_{80}^{-13})}$,
\begin{align}
M =
\begin{pmatrix}
\sqrt{2}A_5  & O \\
O & \frac{1}{\sqrt{2}}A_5
\end{pmatrix},  \quad
N = \begin{pmatrix}
O & A_5  \\
A_5 & O
\end{pmatrix},
\end{align}
where $A_5$ is given in Theorem \ref{thm-G-o}. Hence $X(\tau)$ is a vector-valued modular function on $\Gamma_0(2)$.

The right sides of \eqref{x2-111-mod5} and \eqref{x4-111-mod5} are also the product expressions of those Nahm sums in \cite[Theorem 4.15]{W-rank3}, \cite[Theorem 3.10]{WWg2} and \cite[Theorem 3.19]{WW112} with $q$ replaced by $q^2$.
\end{exam}

\begin{exam}
Recall the Nahm sums in the following identities \cite[Theorem 3.22]{WW122}
\begin{align}
&x_1(q) := \sum_{i,j,k\geq 0}\frac{q^{\frac{1}{2}i^2+5j^2+4k^2+2ij+2ik+8jk+\frac{1}{2}i}}{(q;q)_i (q^2;q^2)_j (q^2;q^2)_k} = \frac{(-q;q)_\infty}{(q^4,q^{16};q^{20})_\infty}, \label{x1-mod20-defn} \\
&x_2(q) := \sum_{i,j,k\geq 0}\frac{q^{\frac{1}{2}i^2+5j^2+4k^2+2ij+2ik+8jk+\frac{1}{2}i+4j+4k}}{(q;q)_i (q^2;q^2)_j (q^2;q^2)_k} = \frac{(-q;q)_\infty}{(q^8,q^{12};q^{20})_\infty}. \label{x2-mod20-defn}
\end{align}
Let
\begin{equation}
X(\tau) := (q^{-1/40}x_1(q), q^{31/40}x_2(q))^{\mathrm{T}}=\frac{\mathfrak{f}_2(\tau)}{\eta(4\tau)} G_{5}^{(1)}(2\tau).
\end{equation}
We have
\begin{align}
X(\tau+1) &= \diag{(\zeta_{40}^{-1}, \zeta_{40}^{31})}X(\tau), \label{X-122-mod20-T}\\
X\left(\frac{\tau}{4\tau+1}\right) &= \frac{4}{5}\begin{pmatrix}
\sin \frac{2\pi}{5} & \sin \frac{\pi}{5} \\
\sin \frac{\pi}{5} & -\sin \frac{2\pi}{5}
\end{pmatrix}
\begin{pmatrix}
\zeta_{10} & 0 \\
0 & \zeta_{10}^{-1}
\end{pmatrix}
\begin{pmatrix}
\sin \frac{2\pi}{5} & \sin \frac{\pi}{5} \\
\sin \frac{\pi}{5} & -\sin \frac{2\pi}{5}
\end{pmatrix}
X(\tau). \label{X-122-mod20-S}
\end{align}
Hence $X(\tau)$ is a vector-valued modular function on $\Gamma_0(4)$.
\end{exam}

\begin{exam}
Recall the following Nahm sums in \cite[Theorem 4.12]{W-rank3}:
\begin{align}
&x_1(q) := \sum_{i,j,k\geq 0}\frac{q^{2i^2+j^2+k^2+2ij+2ik+jk}}{(q;q)_i (q;q)_j (q;q)_k} = \frac{1}{(q,q^4;q^5)_\infty^{2}}, \label{x1-defn}\\
&x_2(q) := \sum_{i,j,k\geq 0}\frac{q^{2i^2+j^2+k^2+2ij+2ik+jk+2i+j+k}}{(q;q)_i (q;q)_j (q;q)_k} = \frac{1}{(q^2,q^3;q^5)_\infty^{2}}, \label{x2-defn}  \\
&x_3(q) := \sum_{i,j,k\geq 0}\frac{q^{2i^2+j^2+k^2+2ij+2ik+jk+i+j}}{(q;q)_i (q;q)_j (q;q)_k} = \frac{(q^5;q^5)_\infty}{(q;q)_\infty}. \label{x3-defn}
\end{align}
Let
\begin{align}
    &X(\tau) := (q^{-1/30}x_1(q), q^{11/30}x_2(q), q^{1/6}x_3(q))^\mathrm{T} \nonumber \\
    &=\frac{1}{\eta^{2}(\tau)}({g_{\frac{1}{2},\frac{5}{2}}^{2}(\tau)},  g_{\frac{3}{2},\frac{5}{2}}^{2}(\tau), g_{\frac{1}{2},\frac{5}{2}}(\tau) g_{\frac{3}{2},\frac{5}{2}}(\tau))^\mathrm{T}.
\end{align}
We have
\begin{align}
X(\tau+1) &= \diag (\zeta_{30}^{-1}, \zeta_{30}^{11}, \zeta_{6})X(\tau), \label{X-mod5-T} \\
X\left(-\frac{1}{\tau}\right) &= \frac{4}{5}
\begin{pmatrix}
\sin^{2} \frac{2\pi}{5} & \sin^{2} \frac{\pi}{5} & 2\sin \frac{\pi}{5} \sin \frac{2\pi}{5}  \\
\sin^{2} \frac{\pi}{5} & \sin^{2} \frac{2\pi}{5} & -2\sin \frac{\pi}{5} \sin \frac{2\pi}{5} \\
\sin \frac{\pi}{5} \sin \frac{2\pi}{5} & -\sin \frac{\pi}{5} \sin \frac{2\pi}{5} & \sin^{2} \frac{\pi}{5}-\sin^{2} \frac{2\pi}{5}
\end{pmatrix}
X(\tau). \label{X-mod5-S}
\end{align}
As a consequence, $X(\tau)$ is a vector-valued modular function on $\mathrm{SL}(2,\mathbb{Z})$.
\end{exam}

\begin{exam}
Bressoud \cite[Eq. (3.8)]{Bressoud1980} proved the following identity:
\begin{align}
&x_{2k,i} (q) := \sum_{n_1\geq n_2\geq \cdots \geq n_{k-1}\geq 0}\frac{(-q^{1-2n_1};q^2)_{n_1}q^{2(n_{1}^2+\cdots +n_{k-1}^{2}+n_i+\cdots +n_{k-1})}}{(q^2;q^2)_{n_1-n_2}(q^2;q^2)_{n_2-n_3}\cdots (q^2;q^2)_{n_{k-1}}}  \nonumber \\
& = \frac{(q^2;q^4)_\infty(q^{2i-1},q^{4k-2i+1},q^{4k};q^{4k})_\infty}{(q;q)_\infty},  \quad 1\leq i\leq k.  \label{B-ori}
\end{align}
Though the left side is no longer a Nahm sum, the product side does express some Nahm sums. For instance, the Nahm sums in \cite[Eqs. (3.5), (3.8)]{Wang2022}, \cite[Eqs. (3.39), (3.43), (3.52), (3.55)]{WW112} and \cite[Eqs. (3.208), (3.211)]{WW122} and \cite[Eqs. (4.57), (4.58)]{W-rank3} all have such product expressions.

Let
\begin{align}
X_{2k}(\tau) &:= \Big( q^{\frac{1}{8k}-\frac{1}{8}}x_{2k,k}(q), q^{\frac{9}{8k}-\frac{1}{8}}x_{2k,k-1}(q), \dots, q^{\frac{(2k-1)^2}{8k}-\frac{1}{8}}x_{2k,1}(q)\Big)^{\mathrm{T}}\nonumber \\
&=\frac{\eta(2\tau)}{\eta(\tau)\eta(4\tau)}G_{2k}^{(1)}(\tau). \label{B-4k}
\end{align}
Applying Theorem \ref{thm-G-o} to the right side of \eqref{B-4k}, we have the following:
\begin{align}
X_{2k}(\tau+1) & = \diag \big( e^{(\frac{1}{4k}-\frac{1}{4})\pi i}, \dots, e^{(\frac{(2k-1)^2}{4k}-\frac{1}{4})\pi i} \big)X_{2k}(\tau), \\
 X_{2k}\big(-\frac{1}{4\tau}\big) &= \sqrt{\frac{2}{k}} A_{2k} X_{2k}(\tau),
\end{align}
where $A_{2k}$ is the same as in Theorem \ref{thm-G-o}.
As a consequence, $X_{2k}(\tau)$ is a vector-valued modular function on $\Gamma_0(4)$.
\end{exam}

\begin{exam}
He--Ji--Wang--Zhao \cite{HJWZ} obtained the following overpartition analogue of Bressoud’s identity \eqref{B-ori}: For  $1\leq i\leq k$ we have
\begin{align}
 &x_{2k-1,i}(q):= \sum_{n_1\geq n_2\geq \cdots \geq n_{k-1}\geq 0}\frac{(-q^{2-2n_1};q^2)_{n_1-1} (-q^{1-2n_1};q^2)_{n_1}(1+q^{2n_{i}})}{(q^2;q^2)_{n_1-n_2}(q^2;q^2)_{n_2-n_3}\cdots (q^2;q^2)_{n_{k-1}}}  \nonumber \\
& \qquad \times q^{2(n_{1}^2+\cdots +n_{k-1}^{2}+n_{i+1}+\cdots +n_{k-1})} = \frac{(-q;q)_\infty(q^{2i-1},q^{4k-1-2i},q^{4k-2};q^{4k-2})_\infty}{(q;q)_\infty}.
\end{align}
Though the sum side is not a Nahm sum, the product side with special $k$ does express some Nahm sums. For instance, the case $k=2$ express Nahm sums in \cite[Eqs. (4.55) and (4.56)]{W-rank3}.

Let
\begin{align}
X_{2k-1}(\tau) &:= \Big( x_{2k-1,k}(q), q^{\frac{1}{2k-1}}x_{2k-1,k-1}(q), \dots, q^{\frac{(k-1)^2}{2k-1}}x_{2k-1,1}(q)  \Big)^{\mathrm{T}} \nonumber \\
&=\frac{\eta(2\tau)}{\eta^2(\tau)}G_{2k-1}^{(0)}(\tau).
\end{align}
We have
\begin{align}
X_{2k-1}(\tau+1) = \diag (1, e^{\frac{2}{2k-1}\pi i}, \dots, e^{\frac{(2k-2)^2}{4k-2}\pi i})X_{2k-1}(\tau),  \\
X_{2k-1}\Big(\frac{\tau}{4\tau+1}\Big) = \frac{2}{2k-1}e^{\frac{\pi i}{2}} B_{2k-1}\widetilde{\Lambda}_{2k-1}C_{2k-1}X_{2k-1}(\tau),
\end{align}
where $B_{2k-1}, \widetilde{\Lambda}_{2k-1}, C_{2k-1}$ are defined in Theorem \ref{thm-G-e}. As a consequence, $X_{2k-1}(\tau)$ is a vector-valued modular function on $\Gamma_0(4)$.
\end{exam}

\begin{exam} We  recall the following two Nahm sums \cite[Eqs. (3.19), (3.22)]{Wang2022}:
\begin{align}
x_1(q):= & \sum_{i,j\geq 0}\frac{q^{i^2 -2ij +2j^2 }}{(q^2 ;q^2)_i (q^2;q^2)_j}= \frac{(q^2;q^2)_{\infty}^3(q^3,q^5,q^8;q^8)_{\infty}}{(q;q)_{\infty}^2(q^4;q^4)_{\infty}^2}, \label{x1-22}  \\
x_2(q):= & \sum_{i,j\geq 0}\frac{q^{i^2 -2ij +2j^2 +2j }}{(q^2 ;q^2)_i (q^2;q^2)_j}= \frac{(q^2;q^2)_{\infty}^3(q,q^7,q^8;q^8)_{\infty}}{(q;q)_{\infty}^2(q^4;q^4)_{\infty}^2}.  \label{x2-22}
\end{align}
Let
\begin{equation}
X(\tau):=(q^{-5/48}x_1(q),q^{19/48}x_2(q))^\mathrm{T}=\frac{\eta^3(2\tau)}{\eta^2(\tau)\eta^2(4\tau)}G_4^{(1)}(\tau).
\end{equation}
We have
\begin{align}
&X(\tau +1)=\diag(\zeta_{48}^{-5}, \zeta_{48}^{19})X(\tau), \label{eg3-1} \\
&X\Big(-\frac{1}{4\tau}\Big) =
\begin{pmatrix}
\sin \frac{3\pi}{8} & \sin \frac{\pi}{8} \\
\sin{\frac{\pi}{8}} & -\sin \frac{3\pi}{8}
\end{pmatrix}
X(\tau) \label{eg3-2} .
\end{align}
As a consequence, $X(\tau)$ is a vector-valued modular function on $\Gamma_0(4)$.
\end{exam}

\begin{exam}
Recall the following Nahm sums from \cite[Eqs. (3.10), (3.14)]{WWg2}:
\begin{align}
x_1(q) := \sum_{i,j\geq 0}\frac{q^{i^2-2ij+2j^2}}{(q^2;q^2)_i (q^4;q^4)_j} = \frac{(q^2;q^2)_\infty^3 (q^3;q^3)_\infty^2}{(q;q)_\infty^2(q^4;q^4)_\infty^2(q^6;q^6)_\infty}, \label{x1-24} \\
x_2(q) := \sum_{i,j\geq 0}\frac{q^{i^2-2ij+2j^2+2j}}{(q^2;q^2)_i (q^4;q^4)_j} = \frac{(q^2;q^2)_\infty^2(q^6;q^6)_\infty^2}{(q;q)_\infty(q^3;q^3)_\infty(q^4;q^4)_\infty^2}.   \label{x2-24}
\end{align}
Let
\begin{equation}
X(\tau) := (q^{-1/6}x_1(q), q^{1/6}x_2(q))^{\mathrm{T}}=\frac{\eta^3(2\tau)}{\eta^2(\tau)\eta^2(4\tau)}G_3^{(0)}(\tau).
\end{equation}
We have
\begin{align}
X(\tau+1) &= \diag (\zeta_{6}^{-1}, \zeta_{6})X(\tau),  \label{X-24-T} \\
X \left(\frac{\tau}{4\tau+1}\right) &=-\frac{\sqrt{3}}{3}\zeta_{12}^{-7}\begin{pmatrix}
    i & 2 \zeta_{12} \\ \zeta_{12} & -\zeta_{12}^{-1}
\end{pmatrix}X(\tau). \label{X-24-S}
\end{align}
As a consequence, $X(\tau)$ is a vector-valued modular function on $\Gamma_0(4)$.
\end{exam}

\begin{exam}
Recall the Nahm sums from \cite[Theorem 4.5]{W-rank3}:
\begin{align}
&x_1(q) := \sum_{i,j,k\geq 0}\frac{q^{2i^2+j^2+k^2+2ij-2ik+k}}{(q^2;q^2)_i (q^2;q^2)_j (q^2;q^2)_k} = \frac{(q^3;q^3)_\infty^2(q^4;q^4)_\infty}{(q;q)_\infty(q^2;q^2)_\infty(q^6;q^6)_\infty}, \label{x1-222}  \\
&x_2(q) :=\frac{1}{2} \sum_{i,j,k\geq 0}\frac{q^{2i^2+j^2+k^2+2ij-2ik+2i+2j-k}}{(q^2;q^2)_i (q^2;q^2)_j (q^2;q^2)_k} = \frac{(q^4;q^4)_\infty(q^6;q^6)_\infty^2}{(q^2;q^2)_\infty^2(q^3;q^3)_\infty}, \label{x2-222} \\
&x_3(q) := \sum_{i,j,k\geq 0}\frac{q^{2i^2+j^2+k^2+2ij-2ik+j}}{(q^2;q^2)_i (q^2;q^2)_j (q^2;q^2)_k} = \frac{(q^2;q^2)_\infty^2(q^3;q^3)_\infty^2}{(q;q)_\infty^2(q^4;q^4)_\infty(q^6;q^6)_\infty},  \label{x3-222} \\
&x_4(q) := \sum_{i,j,k\geq 0}\frac{q^{2i^2+j^2+k^2+2ij-2ik+2i+j}}{(q^2;q^2)_i (q^2;q^2)_j (q^2;q^2)_k} = \frac{(q^2;q^2)_\infty(q^6;q^6)_\infty^2}{(q;q)_\infty(q^3;q^3)_\infty(q^4;q^4)_\infty}. \label{x4-222}
\end{align}
Let
\begin{align}
X_1(\tau) := (q^{1/24}x_1(q), q^{3/8}x_2(q))^{\mathrm{T}} = \frac{\eta(4\tau)}{\eta(\tau)\eta(2\tau)}G_{3}^{(0)}(\tau), \\
X_2(\tau) := (q^{-1/12}x_3(q), q^{1/4}x_4(q))^{\mathrm{T}} = \frac{\eta^{2}(2\tau)}{\eta^{2}(\tau)\eta(4\tau)}G_{3}^{(0)}(\tau).
\end{align}
We have
\begin{align}
 X_1(\tau+1) &= \diag (\zeta_{24}, \zeta_{8}^{3})X_1(\tau),  \\
 X_2(\tau+1) &= \diag (\zeta_{12}^{-1}, \zeta_{4})X_2(\tau),  \\
X_1\left(\frac{\tau}{4\tau +1}\right) & = \frac{1}{3}
\begin{pmatrix}
2\zeta_8+\zeta_{24}^{11} & 2\zeta_8-2\zeta_{24}^{11} \\
\zeta_8-\zeta_{24}^{11} & \zeta_8+2\zeta_{24}^{11}
\end{pmatrix}
X_1(\tau), \\
X_2\left(\frac{\tau}{4\tau +1}\right) & = \frac{1}{3}
\begin{pmatrix}
2\zeta_8+\zeta_{24}^{11} & 2\zeta_8-2\zeta_{24}^{11} \\
\zeta_8-\zeta_{24}^{11} & \zeta_8+2\zeta_{24}^{11}
\end{pmatrix}
X_2(\tau).
\end{align}
As a consequence, $X_i(\tau)$ ($i=1,2$) is a vector-valued modular function on $\Gamma_0(4)$.
\end{exam}

\begin{exam}
Recall the following Nahm sums from \cite[Theorem 4.14]{W-rank3}:
\begin{align}
x_1(q) := \sum_{i,j,k\geq 0}\frac{q^{4i^2+2j^2+k^2+4ij-2ik-2jk}}{(q^2;q^2)_i (q^2;q^2)_j (q^2;q^2)_k} = \frac{(q^2;q^2)_\infty^{3} (q^5,q^7,q^{12};q^{12})_\infty}{(q;q)_\infty^{2} (q^4;q^4)_\infty^{2}}, \label{x1-mod12-222} \\
x_2(q) := \sum_{i,j,k\geq 0}\frac{q^{4i^2+2j^2+k^2+4ij-2ik-2jk+2i}}{(q^2;q^2)_i (q^2;q^2)_j (q^2;q^2)_k} = \frac{(q^2;q^2)_\infty^{3} (q^3,q^9,q^{12};q^{12})_\infty}{(q;q)_\infty^{2} (q^4;q^4)_\infty^{2}}, \label{x2-mod12-222}  \\
x_3(q) := \sum_{i,j,k\geq 0}\frac{q^{4i^2+2j^2+k^2+4ij-2ik-2jk+4i+2j}}{(q^2;q^2)_i (q^2;q^2)_j (q^2;q^2)_k} = \frac{(q^2;q^2)_\infty^{3} (q,q^{11},q^{12};q^{12})_\infty}{(q;q)_\infty^{2} (q^4;q^4)_\infty^{2}}. \label{x3-mod12-222}
\end{align}
Let
$$X(\tau) := (q^{-1/8}x_1(q), q^{5/24}x_2(q), q^{7/8}x_3(q))^{\mathrm{T}}=\frac{\eta^3(2\tau)}{\eta^2(\tau)\eta^2(4\tau)}G_6^{(1)}(\tau).$$
We have
\begin{equation}
X(\tau +1) = \diag (\zeta_{8}^{-1}, \zeta_{24}^{5}, \zeta_{8}^{7})X(\tau)
\end{equation}
and
\begin{equation}
X\left(-\frac{1}{4\tau}\right) = \sqrt{\frac{2}{3}}
\begin{pmatrix}
\sin \frac{5\pi}{12} & \sin \frac{\pi}{4} & \sin \frac{\pi}{12}  \\
\sin \frac{\pi}{4} & -\sin \frac{\pi}{4} & -\sin \frac{\pi}{4}  \\
\sin \frac{\pi}{12} & -\sin \frac{\pi}{4} & \sin \frac{5\pi}{12}
\end{pmatrix}
X(\tau).
\end{equation}
As a consequence, $X(\tau)$ is a vector-valued modular function on $\Gamma_0(4)$.
\end{exam}

\begin{exam}
Recall the Nahm sums from \cite[Theorem 3.10]{WW112}:
\begin{align}
x_1(q) := \sum_{i,j,k\geq 0}\frac{q^{\frac{1}{2}i^2+2j^2+k^2+ij+2jk+\frac{1}{2}i}}{(q;q)_i (q;q)_j (q^2;q^2)_k} = \frac{(-q;q)_\infty}{(q,q^4,q^7;q^8)_\infty},  \label{x1-mod8-112} \\
x_2(q) := \sum_{i,j,k\geq 0}\frac{q^{\frac{1}{2}i^2+2j^2+k^2+ij+2jk+\frac{1}{2}i+2j+2k}}{(q;q)_i (q;q)_j (q^2;q^2)_k} = \frac{(-q;q)_\infty}{(q^3,q^4,q^5;q^8)_\infty}.  \label{x2-mod8-112}
\end{align}
Let
\begin{equation}
X(\tau) := (q^{-1/48}x_1(q), q^{23/48}x_2(q))^{\mathrm{T}} = \frac{\eta^2(2\tau)}{\eta^2(\tau)\eta(4\tau)}G_{4}^{(1)}(\tau).
\end{equation}
We have
\begin{align}
X(\tau+1) &= \diag (\zeta_{48}^{-1}, \zeta_{48}^{23})X(\tau), \\
X \left(\frac{\tau}{4\tau+1}\right)
&=\frac{\sqrt{2}}{2}\zeta_{48}^7\begin{pmatrix}
    1 & 1 \\ 1 & -1
\end{pmatrix}X(\tau).
\end{align}
As a consequence, $X(\tau)$ is a vector-valued modular function on $\Gamma_0(4)$.
\end{exam}

The above examples show that a number of Nahm sums form vector-valued modualr functions on $\Gamma_0(N)$ with $N\in \{1,2,3,4\}$. However, this might not always be true as can be seen from the following example.
\begin{exam}
Recall the identities in \cite[Theorem 3.9]{WWg2}:
\begin{align}
x_1(q) := \sum_{i,j\geq 0}\frac{q^{2i^2-4ij+3j^2-2i+2j}}{(q^4;q^4)_i (q^8;q^8)_j} = \frac{(-1;q^4)_\infty (q^2,q^3,q^5;q^5)_\infty}{(q,q^3,q^4;q^4)_\infty}, \label{x1-48} \\
x_2(q) := \sum_{i,j\geq 0}\frac{q^{2i^2-4ij+3j^2-2i+4j}}{(q^4;q^4)_i (q^8;q^8)_j} = \frac{(-1;q^4)_\infty (q,q^4,q^5;q^5)_\infty}{(q,q^3,q^4;q^4)_\infty}. \label{x2-48}
\end{align}
Let
\begin{equation}
X(\tau) := (q^{1/15}x_1(q), q^{4/15}x_2(q))^{\mathrm{T}}=2\frac{\mathfrak{f}_2(4\tau)\eta(2\tau)}{\eta(\tau)\eta(4\tau)} G_{5}^{(1)}\Big(\frac{\tau}{2}\Big).
\end{equation}
We have
\begin{align}
 X(\tau+1) &= \diag (\zeta_{15}, \zeta_{15}^{4})X(\tau), \label{X-48-mod5-T} \\
 X \left(\frac{\tau}{8\tau+1}\right) &=  \frac{4}{5} A_5 P_1 A_5 X(\tau).  \label{X-48-mod5-S}
\end{align}
Here $A_5$ is defined in Theorem \ref{thm-G-o} and $P_1 = \diag (\zeta_{120}^{11}, \zeta_{120}^{-181})$.
As a consequence, $X(\tau)$ is modular on a  subgroup of $\Gamma_0(4)$ generated by the matrices $\pm \left(\begin{smallmatrix}
    1 & 0 \\ 0 & 1
\end{smallmatrix}\right)$, $ \left(\begin{smallmatrix}
    1 & 1 \\ 0 & 1
\end{smallmatrix}\right)$ and $\left(\begin{smallmatrix}
    1 & 0 \\ 8 & 1
\end{smallmatrix}\right)$. However,
\begin{equation}
X \left(\frac{\tau}{4\tau+1}\right) = \frac{4\sqrt{2}}{5} A_5 P_2 A_5
\frac{\mathfrak{f}(4\tau)\eta(2\tau)}{\eta(\tau)\eta(4\tau)} G_{5}^{(1)}\big(\frac{\tau}{2}\big)
\end{equation}
where $P_2 = \diag (\zeta_{240}^{11}, \zeta_{240}^{-181})$.
This shows that $X(\tau)$ is not modular on $\Gamma_0(4)$. We may need to find more modular Nahm sums associated with the same matrix to enlarge the group.
\end{exam}


\begin{thebibliography}{0}

\bibitem{Andrews1974} G.E. Andrews, On the general Rogers--Ramanujan theorem, Providence, RI: Amer. Math. Soc., 1974.

\bibitem{Berndt-book} B.C. Berndt, Number theory in the spirit of Ramanujan, Amer. Math. Soc., 2006.



\bibitem{Bressoud1980} D.M. Bressoud, Analytic and combinatorial generalizations of Rogers--Ramanujan identities, Mem. Amer. Math. Soc. 24(227)(1980), 54 pp.

\bibitem{C-R-W} Z. Cao, H. Rosengren and L. Wang, On some double Nahm sums of Zagier, J. Combin. Theory Ser. A 202 (2024), 105819.

\bibitem{Capparelli} S. Capparelli, On some representations of twisted affine Lie algebras and combinatorial identities, J. Algebra, 154(2) (1993), 335--355.


\bibitem{Feigin} I. Cherednik and B. Feigin, Rogers--Ramanujan type identities and Nil-DAHA, Adv. Math. 248 (2013), 1050--1088.


\bibitem{Garvan-Liang} J. Frye and F.G. Garvan, Automatic proof of theta-function identities, Elliptic integrals, elliptic
functions and modular forms in quantum field theory, Texts Monogr. Symbol. Comput., Springer,
Cham, 2019, 195--258.

\bibitem{Garvan} F.G. Garvan, The crank of partitions mod 8, 9 and 10, Trans. Am. Math. Soc. 322(1) (1990), 79--94.

\bibitem{HJWZ} T.Y. He, K.Q. Ji, A.Y.F. Wang and A.X.H. Zhao, An overpartition analogue of the Andrews--Göllnitz--Gordon theorem, J. Combin. Theory Ser. A 166 (2019), 254--296.

\bibitem{K-R-conj} S. Kanade and M. C. Russell, Identityfinder and some new identities of Rogers–Ramanujan type, Exp.  Math., 24(4)(2015), 419--423.

\bibitem{KR2019} S. Kanade and M.C. Russell, Staircases to analytic sum-sides for many new integer partition identities of Rogers--Ramanujan type, Electron. J. Comb. 26(1) (2019), 1--6.

\bibitem{Kursungoz} K. Kur\c{s}ung\"oz, Andrews--Gordon type series for Capparelli's and G\"ollnitz--Gordon identities, J. Combin. Theory, Ser. A, 165 (2019), 117--138.

\bibitem{Kursungoz-2} K. Kur\c{s}ung\"oz, Andrews--Gordon type series for Kanade--Russell conjectures, Ann. Comb. 23 (2019), 835--888.

\bibitem{Milas-Wang} A. Milas and L. Wang, Modularity of Nahm sums for the tadpole diagram, Int. J. Number Theory, 20(1) (2024), 73--101.


\bibitem{Mizuno-note} Y. Mizuno, Modular aspects of the $q$-series in the Kanade--Russell mod 9 conjecture, in preparation.

\bibitem{Mizuno} Y. Mizuno, Remarks on Nahm sums for symmetrizable matrices,  arXiv:2305.02267v1.

\bibitem{Nahm-question} W. Nahm, Conformal field theory and the dilogarithm, in: 11th International Conference on Mathematical Physics (ICMP-11) (Satellite colloquia: New Problems in the General Theory of Fields and Particles), Paris, 1994, 662--667.

\bibitem{Nahm-conj} W. Nahm, Conformal field theory and torsion elements of the Bloch group, in: Frontiers in Number Theory, Physics and Geometry, II, Springer, 2007, 67--132.

\bibitem{R-R-Rogers} L.J. Rogers, Second memoir on the expansion of certain infinite products, Proc. London Math. Soc. 25 (1894), 318--343.

\bibitem{VZ} M.\ Vlasenko and S.\ Zwegers, Nahm's conjecture: asymptotic computations and counterexamples, Commun. Number theory Phys. 5(3) (2011), 617--642.

\bibitem{Wakimoto} M.\ Wakimoto, Infinite-dimensional Lie algebras, Vol. 195. Amer. Math. Soc., 2001.

\bibitem{WWg2} B. Wang and L. Wang, Proofs of Mizuno's conjectures on generalized rank two Nahm sums, arXiv:2308.14728v1.

\bibitem{WW112} B. Wang and L. Wang, Mizuno's rank three Nahm sums I: identities of index (1,1,2), arXiv:2402.06253v1.

\bibitem{WW122} B. Wang and L. Wang, Mizuno's rank three Nahm sums II: identities of index (1,2,2) and modular forms, arXiv:2407.21725v1.

\bibitem{Wang2022} L. Wang, Identities on Zagier's rank two examples for Nahm's problem, Res Math Sci 11, 49 (2024).

\bibitem{W-rank3} L. Wang, Explicit forms and proofs of Zagier's rank three examples for Nahm's problem, Adv. Math. 450 (2024), 109743.

\bibitem{We}  H. Weber, Lehrbuch der Algebra, Bd.3, Elliptische Funktionen und Algebraische
Zahlen, Braunschweig, 1908.

\bibitem{Yang} Y. Yang, Transformation formulas for generalized Dedekind eta functions, Bull. London Math. Soc, 36 (2004), 671--682.

\bibitem{Zagier2007} D. Zagier, The dilogarithm function, in: Frontiers in Number Theory, Physics and Geometry, II, Springer, 2007, 3--65.
\end{thebibliography}
\end{document}